\documentclass[a4,reqno,12pt]{amsart}
\usepackage{amsmath, amsfonts, amssymb, amsthm, amscd}
\usepackage[pdftex]{graphicx}
\usepackage[pdftex,dvipsnames]{xcolor}
\usepackage{mathtools}

\usepackage[colorlinks]{hyperref}
\DeclareGraphicsExtensions{.ps,.eps}

\usepackage[latin1]{inputenc}
\usepackage{bbm}

\numberwithin{equation}{section}

\newtheorem{theorem}{Theorem}[section]
\newtheorem{corollary}{Corollary}[section]

\newtheorem{lemma}{Lemma}[section]
\newtheorem{proposition}{Proposition}[section]

\newtheorem{remark}{Remark}

\newtheorem{bigthm}{Theorem}   

\newcommand{\sumtwo}[2]{\sum_{\substack{#1 \\ #2}}} 
\newcommand{\abs}[1]{\left| #1\right|}

\newcommand{\calA}{\mathcal{A}}

\newcommand{\calD}{\mathcal{D}}
\newcommand{\calE}{\mathcal{E}}
\newcommand{\calF}{\mathcal{F}}

\newcommand{\calH}{\mathcal{H}}

\newcommand{\calM}{\mathcal{M}}

\newcommand{\calP}{\mathcal{P}}

\newcommand{\calS}{\mathcal{S}}

\newcommand{\frD}{\mathfrak{D}}

\newcommand{\frG}{\mathfrak{G}}

\newcommand{\bbA}{\mathbb{A}}

\newcommand{\bbE}{\mathbb{E}}

\newcommand{\bbL}{\mathbb{L}}

\newcommand{\bbN}{\mathbb{N}}

\newcommand{\bbP}{\mathbb{P}}
\newcommand{\bbQ}{\mathbb{Q}}
\newcommand{\bbR}{\mathbb{R}}

\newcommand{\bbX}{\mathbb{X}}

\newcommand{\bbZ}{\mathbb{Z}}

\newcommand{\sfe}{\mathsf e}

\newcommand{\sfp}{{\mathsf p}}

\newcommand{\sfw}{{\mathsf w}}

\newcommand{\sfC}{\mathsf{C}}
\newcommand{\sfD}{\mathsf{D}}

\newcommand{\sfK}{\mathsf{K}}
\newcommand{\sfL}{\mathsf{L}}

\newcommand{\sfS}{\mathsf{S}}
\newcommand{\sfT}{\mathsf{T}}

\newcommand{\ur}{\underline{r}}
\newcommand{\us}{\underline{s}}

\newcommand{\uv}{\underline{v}}
\newcommand{\uu}{\underline{u}}
\newcommand{\ux}{\underline{x}}
\newcommand{\uy}{\underline{y}}
\newcommand{\uz}{\underline{z}}
\newcommand{\uw}{\underline{w}}

\newcommand{\uB}{\underline{B}}
\newcommand{\uX}{\underline{X}}

\newcommand{\lb}{\left(}
\newcommand{\rb}{\right)}

\newcommand{\dd}{{\rm d}}

\newcommand{\step}[1]{S{\small TEP}\,#1.}

\newcommand{\1}{\mathbbm{1}}

\newcommand{\smo}[1]{{\mathrm o}\lb #1\rb }

\newcommand{\df}{\stackrel{\Delta}{=}}

\newcommand{\be}[1]{\begin{equation}\label{#1}}
\newcommand{\ee}{\end{equation}}

\newcommand{\Hla}{H_\lambda}

\newcommand{\An}{\bbA_n^+}
\newcommand{\Anl}{\bbA_{n, \lambda}^{+}}
\newcommand{\Anr}{\bbA_n^{+, {\mathsf r}}}
\newcommand{\Anlr}{\bbA_{n, \lambda}^{+, {\mathsf r}}}

\begin{document}
\date{\today} 

\title[Ordered walks under area tilts]
{Dyson Ferrari--Spohn diffusions and ordered walks under area tilts}

\author{Dmitry Ioffe}
\address{Faculty of IE\&M, Technion, Haifa 32000, Israel}
\email{ieioffe@ie.technion.ac.il}
\thanks{DI was supported by the Israeli Science Foundation grant 
1723/14.}

\author{Yvan Velenik}
\address{Section de Math\'ematiques, Universit\'e de Gen\`eve, 
CH-1211 Gen\`eve, Switzerland}
\email{yvan.velenik@unige.ch}
\thanks{YV was partially supported by the Swiss National Science Foundation.}

\author{Vitali Wachtel}
\address{Institut f\"ur Mathematik, Universit\"at Augsburg, D-86135 Augsburg, 
Germany}
\email{vitali.wachtel@math.uni-augsburg.de}

\maketitle

\begin{abstract}
We consider families of non-colliding random walks above a hard wall,
which 
are subject to a self-potential of tilted area type. 
We view such ensembles as effective models for the level lines of a class of 
$2+1$-dimensional discrete-height random surfaces in statistical mechanics.  
We prove that, under rather general assumptions on the step distribution and on 
the self-potential, such walks converge, under appropriate rescaling, to 
non-intersecting Ferrari--Spohn diffusions associated with limiting 
Sturm--Liouville operators. In particular, the limiting invariant measures are 
given by the squares of the corresponding Slater determinants. 
\end{abstract}
\section{Introduction} 
Random walks under area tilts mimic phase separation lines in certain 
low-temperature two-dimensional lattice models of statistical mechanics,  
particularly in the regime of pre-wetting.  A prototypical example is the 
two-dimensional Ising model in a large box with negative boundary conditions
and a small positive magnetic field $h$. In such circumstances, the 
$\pm$-interface is pushed towards the boundary of the box and its fluctuations 
above flat segments of the boundary are expected to be of order $h^{-1/3}$.
Rigorous justification of the latter claim is still an open problem (but 
see~\cite{V04} for partial results in this direction). Instead, the 
papers~\cite{HrynivVelenik09,IofShlosVel2014} are devoted to a refined 
analysis of effective random walk models of such interfaces. In particular, 
the full scaling limits were identified in~\cite{IofShlosVel2014}, for a large 
class of effective random walks, as Ferrari--Spohn 
diffusions~\cite{FerSpohn05}. 

\smallskip
In this paper, we consider ensembles of $n$ non-colliding random walks 
which are  subject to generalized area tilts. Precise definition are given in
Section~\ref{sec:RW}. These ensembles are intended to model non-intersecting 
level lines for certain low-temperature $2+1$-dimensional interfaces (which 
themselves are intended to model two-dimensional random surfaces of lattice
statistical mechanics). A prototypical example is the SOS model, 
see~\cite{CapLubMarSlyTon13,CapMarTon14} and references therein, or even more 
so its version with bulk Bernoulli fields which was introduced 
in~\cite{IofShlos08}. In either case, low-temperature level lines have the 
structure of Ising polymers whose effective random walk representation is 
discussed in~\cite{IofShlosTon15} and is based on the general fluctuation 
theory of ballistic walks with self-interactions as developed 
in~\cite{IofVel08}. 

\smallskip
In Section~\ref{sec:S-T}, we introduce and briefly discuss the class of limiting 
objects, which we call Dyson Ferrari--Spohn diffusions. The latter can be 
alternatively described as Ferrari--Spohn diffusions conditioned to remain
ordered, or as ergodic $n$-dimensional diffusions driven by 
the $\log$-derivative of the Slater determinants of the corresponding 
Sturm--Liouville operators.
The construction is well understood: 
We refer to~\cite[Section~2]{Sosh00} for extensive details on determinantal 
random point fields in general and Fermi gas in particular, and 
to~\cite[Section~3]{Dui15} where such diffusions are discussed for specific 
kernels in the context of random matrix theory.  

Properties of Dyson Ferrari--Spohn diffusions, in the form we need them, are 
formulated in Theorem~\ref{thm:Vdm}.
To keep our exposition self-contained and to stress the role played by the 
Karlin--McGregor formula, we sketch the proof. 

\smallskip
Our effective model of ordered walks under area tilts is introduced in 
Section~\ref{sec:RW}, while our main result, Theorem~\ref{thm:A}, is formulated 
in Subsection~\ref{sub:Main}. In Subsection~\ref{sub:resc-notation}, we 
introduce the rescaling notation which is employed in all the subsequent 
arguments. The step-by-step structure of our arguments is explained in 
Subsection~\ref{sub:Structure}. The details of the proofs are given in 
Sections~\ref{sec:CM-Prob}--\ref{sec:lems} and in the Appendix. The 
organization of these sections is described in Subsection~\ref{sub:Org}. Many 
of our technical estimates rely on strong approximation techniques and on a 
refinement of recent results on random walks in Weyl chambers and on 
cones~\cite{DW15, DW15b}.

\section{Sturm--Liouville operators and Dyson Ferrari--Spohn  diffusions}
\label{sec:S-T}
\subsection{One particle} 

Consider the Sturm--Liouville operator
\be{eq:SL}
\sfL = \frac{1}{2}\frac{\dd^2}{\dd r^2} - q (r ) ,
\ee
where $q$ is a non-negative symmetric $\sfC^2$-potential satisfying 
\be{eq:q-cond}
\lim_{\abs{r}\to\infty} q(r)=\infty .
\ee
We can either think of $\sfL$ as being defined on $\bbL_2(\bbR_+)$
with zero boundary conditions at zero, or as being defined on $\bbL_2(\bbR)$. 
It is classical fact~\cite{CodLev55} 
that $\sfL$ has a complete orthonormal family of eigenfunctions
\be{eq:eigenphi}
\sfL \varphi_i = -\sfe_i\varphi_i\quad 0 <\sfe_1 <\sfe_2 <\dots \nearrow\infty .
\ee
The Krein--Rutman eigenfunction $\varphi_1$ is positive on $(0,\infty)$, 
respectively on $\bbR$. 

In the case of the half-line, $\sfL$ has a closed self-adjoint extension 
from $\sfC_0(0,\infty)$, whereas in the case of $\bbR$ it has a closed
self-adjoint extension from $\sfC_0(\bbR)$. In both cases, the domain of
the closure is given by 
\be{eq:L-domain}
\calD(\sfL) =
\Bigl\{ f=\sum_k f_k \varphi_k~:~ \sum_k \sfe_k f_k^2 <\infty \Bigr\} .
\ee
We proceed to discuss the half-line case only; the full-line case would be
a literal repetition. 

$\sfL$ is a generator of a contraction semigroup $\sfT_t$ on $\bbL_2(\bbR_+)$:
For $f = \sum f_k\varphi_k$, 
\be{eq:Tt}
\sfT_t f (r) = \sum_k f_k \mathrm{e}^{-\sfe_k t}  \varphi_k .
\ee
This semigroup has the following probabilistic representation: For $r>0$, 
let $\hat{\mathbf P}^r$ be the (sub-probability) path measure of the Brownian 
motion started at $r$ and killed upon hitting the origin. Then, for any 
$f\in \calD(\sfL)$ and any $t\geq 0$, 
\be{eq:Tt-BM}
\sfT_t f (r) = \hat{\mathbf E}^r \bigl\{ \mathrm{e}^{-\int_0^t q (B(s))\,\dd s} 
f (B (t))\bigr\} .
\ee
Clearly, $\sfT_t $ is an integral operator with kernel $h_t$ given by 
\be{eq:ht}
h_t (r,s) = \sum \mathrm{e}^{-\sfe_k t}\varphi_k (r)\varphi_k (s) .
\ee

\subsection{$n$ non-colliding particles.} 

Let us fix  $n\in\bbN$ and define
\be{eq:Aplus}
\bbA_n^+ = \{ \ur \in \bbR^n \,:\, 0< r_1 <\dots < r_n\} .
\ee
Let $\sfL_i$ be a copy of $\sfL$ acting on the $i$-th variable. Consider 
the closed self-adjoint extension of 
\be{eq:frL}
\sum_1^n \sfL_i =
\sum_1^n \Bigl( \frac{1}{2}\frac{\partial^2}{\partial r_i^2} - q (r_i) \Bigr) 
\ee
from $\sfC_0(\bbA_n^+)$ and let $\sfT_t^{+}$ be the corresponding contraction 
semigroup on $\bbL_2(\bbA_n^+)$. 

In probabilistic terms, $\sfT_t^{+}$ can be described as follows: For an 
$n$-tuple $\ur\in\bbA_n^+$, set 
\[
\hat{\mathbf E}^{\ur} =
\hat{\mathbf E}^{r_1} \otimes \hat{\mathbf E}^{r_2} \otimes \cdots \otimes 
\hat{\mathbf E}^{r_n} .
\]
Let $\uB$ be the $n$-dimensional Brownian motion, and define  
\be{eq:Tkilled}
\tau  = \min\{ t \,:\, \uB(t)\not\in\bbA_n^+ \} .
\ee
In other words, $\tau$ is the minimum between the first collision time and
the first time the bottom trajectory exits from the positive semi-axis. Then, 
\be{eq:Tt-plus}
\sfT_t^{+} f (\ur) =
\hat{\mathbf E}_{\ur} \bigl\{ \mathrm{e}^{-\sum_i \int_0^t q (B_i(s))\,\dd s}
f(\uB(t)) \1_{\tau > t} \bigr\} .
\ee
$\sfT_t^+$ is an integral operator on $\bbL_2(\bbA_n^+)$ and, by the 
Karlin--McGregor formula, its kernel $\kappa_t$ is given by 
\be{eq:kappa-t}
\kappa_t (\ur,\us) = \operatorname{det} \{ h_t(r_i,s_j) \} .
\ee

\subsection{Limiting behaviour.} 

Let
\be{eq:DeltaVdm}
\Delta (\ur) 
= 
\operatorname{det} 
\begin{bmatrix}
\varphi_1(r_1)	& \varphi_2(r_1) 	& \cdots &\varphi_n (r_1 ) \\
\varphi_1(r_2)	& \varphi_2(r_2) 	& \cdots &\varphi_n (r_2 ) \\
\vdots 		& \vdots 		& \ddots &\vdots \\
\varphi_1(r_n)	& \varphi_2(r_n) 	& \cdots &\varphi_n (r_n ) 
\end{bmatrix} .
\ee
Note that $\Delta\in\bbL_2(\bbA_n^+)$ .

\begin{theorem}
\label{thm:Vdm}
Set $\sfD_n  = \sum_1^n\sfe_\ell$. Then, 
\be{eq:Vdm} 
\lim_{t\to\infty} \mathrm{e}^{\sfD_n t} \kappa_t (\ur,\us)
= 
\Delta(\ur) \Delta(\us) .
\ee
Moreover, for any
$f\in\bbL_2(\bbA_n^+)$ and for any $t>0$,
\begin{align}
&\lim_{T\to\infty}
\frac{\bbE_{\ur} \bigl\{ \mathrm{e}^{-\sum_i \int_0^T q(B_i(s))\,\dd s}
f(\uX(t))
\1_{T < \tau_+} \bigr\} }
{\bbE_{\ur}\bigl\{ \mathrm{e}^{-\sum_i \int_0^T q(B_i(s))\,\dd s}
\1_{T < \tau_+} \bigr\} } \notag\\ 
&\hspace{3cm} = 
\frac{\mathrm{e}^{ \sfD_n t}}{\Delta (\ur )}
\bbE_{\ur}\bigl\{ \mathrm{e}^{-\sum_i \int_0^t q (B_i (s))\,\dd s}
f(\uB(t)) \Delta(\uB(t)) \1_{t < \tau_+} \bigr\} \notag\\
&\hspace{3cm} = 
\frac{\mathrm{e}^{ \sfD_n t}}{\Delta(\ur)} \sfT_t^+(f\Delta)(\ur)
\df
\sfS_t^+ f (\ur) .
\label{eq:Vdm-lim}
\end{align}
In its turn, $\sfS_t^+$ is a diffusion semigroup with transition kernel 
\be{eq:qt-kernel}
q_t (\ur,\us) = 
\frac{\mathrm{e}^{\sfD_n t }}{\Delta(\ur)} \kappa_t(\ur,\us) \Delta(\us) ,
\ee
which is self-adjoint on $\bbL_2(\bbA_n^+,\Delta^2)$; the generator of the 
corresponding ergodic diffusion on $\bbA_n^+$ is given by 
\be{eq:CondForm2} 
\frG_{n}^+
= 
\frac{1}{2}\sum_1^n\frac{\partial^2}{\partial r_i^2} + 
\nabla\log (\Delta) (\ur) 
\cdot\nabla 
=
\frac{1}{2 \Delta^2(\ur)} {\rm div} \bigl(\Delta^2(\ur)\nabla\bigr) .
\ee
\end{theorem}
\begin{proof}
Let us introduce the column vectors
\[
b_\ell = 
\begin{bmatrix}
  \varphi_\ell (r_1 )\\
  \varphi_\ell (r_2 )\\
  \vdots \\
  \varphi_\ell (r_n )
\end{bmatrix}
\qquad
(\ell=1, \ldots , n)
\]
and the volume form $F(c_1,\ldots,c_n) = \operatorname{det}[c_1,\ldots,c_n]$.
Under our assumptions, 
\eqref{eq:eigenphi}, \eqref{eq:ht} and~\eqref{eq:kappa-t} imply, asymptotically 
as $t\to\infty$, that
\begin{align}
\mathrm{e}^{\sfD_n t} \kappa_t(\ur,\us) \bigl( 1+ \smo{1} \bigr)
&= F\Bigl(
\sum_1^n \varphi_\ell(s_1) b_\ell, 
\sum_1^n \varphi_\ell(s_2) b_\ell , \ldots ,
\sum_1^n \varphi_\ell(s_n) b_\ell 
\Bigr)
\notag\\
&= 
F (b_1,\ldots,b_n) \sum_\sigma (-1)^{\operatorname{sgn}(\sigma)} 
\prod_{\ell=1}^n \varphi_{\sigma_\ell}(s_\ell) ,
\label{eq:F-form}
\end{align}
and the first claim~\eqref{eq:Vdm} follows. Above, $\sigma$ runs over all 
permutations of $\{1,\ldots,n\}$ and $\operatorname{sgn}(\sigma) = \pm 1$ 
denotes the signature of $\sigma$. 

Modulo some technicalities, \eqref{eq:Vdm-lim} follows from~\eqref{eq:Vdm} and 
the Markov property. 

Finally, $\frG_n$ in~\eqref{eq:CondForm2} is the generator of $\sfS_t^+$, since 
the generator of $\sfT_t^+$ is the closed self-adjoint extension 
of~\eqref{eq:frL} from $\sfC_0(\bbA_n^+)$ and since, by direct computation, 
\be{eq:frL-Delta} 
\sum_1^n \Bigl( \frac{1}{2}\frac{\partial^2}{\partial r_i^2} - q (r_i )\Bigr)
\Delta(\ur)
= -\sfD_n \Delta(\ur) .
\ee
\end{proof}

\subsection{Dyson diffusions for Sturm--Liouville operators.} 

For every $n\in\bbN$, the diffusion
\be{eq:FS-diffusion}
\dd\ux (t) = \dd\uB (t) + \nabla\log(\Delta) (\ux(t)) \,\dd t ,
\ee
with the generator $\frG_n$ described in Theorem~\ref{thm:Vdm}, lives on 
$\bbA_n^+$ and is reversible with respect to $\Delta^2(\ur)\,\dd\ur$. In the 
sequel, we shall use $\bbP_n^+$ for its distribution on 
$\sfC\bigl((-\infty,+\infty),\bbA_n^+\bigr)$ and $\bbP_n^{+;T}$ for the 
restriction of this distribution to $\sfC\bigl([-T,T],\bbA_n^+\bigr)$. 
Without loss of generality, let us assume that $\Delta^2(\ur)$ is a probability 
density (on $\bbA_n^+$). Note that the latter has a determinantal structure: 
\be{eq:n-Kernel-det}
\Delta^2(\ur) = 
\operatorname{det} \{ \sfK_n (r_i,r_j) \} ,
\ee
where the kernel $K_n$ is given by
\be{eq:n-Kernel}
\sfK_n (r,s) = \sum_{\ell=1}^n \varphi_\ell(r) \varphi_l(s) .
\ee
In particular, the level density distribution is given by
\be{eq:level-dens}
\rho_n (r) = \frac{1}{n} K_n(r,r) = \frac{1}{n} \sum_1^n \varphi_\ell^2(r) .
\ee
There are similar determinantal formulas for level spacing, gap probabilities, 
etc. We note that the unpublished work~\cite{Bornemann2012} contains results 
on the universality of scaling limits (as $n\to\infty$) in this general 
Sturm--Liouville context. 
 
\section{Ordered walks with area tilts.}
\label{sec:RW}
\subsection{Underlying random walks and ordering of trajectories}
The setup follows~\cite{IofShlosVel2014}. 

Let $p_y$ be an irreducible random walk kernel on $\bbZ$.
The probability of a finite trajectory $\bbX = (X_1,X_2,\ldots,X_k)$ is
$\sfp(\bbX) = \prod_i p_{X_{i+1} - X_i}$.
The product probability of $n$ finite trajectories $\underline{\bbX} = 
(\bbX^1,\ldots,\bbX^n)$ is
\be{eq:RW-measure}
\mathbf{P} (\underline{\bbX}) = \prod_{\ell=1}^n \sfp (\bbX^\ell). 
\ee

\paragraph{\emph{Assumptions on $\sfp$}}
Assume that
\begin{equation}
\label{eq:Assumption}
\sum_{z\in\bbZ} z p_z = 0
\quad\text{and $\sfp$ has finite exponential moments}.
\end{equation}
In the sequel, we, 
{
in order to facilitate the notation, 
}
shall assume
that the variance satisfies, 
\be{eq:sigma}
\sigma^2 = \sum_{z\in\bbZ} z^2 p_z
= 1 .
\ee

\smallskip
\paragraph{\emph{Sets of trajectories}}

 
Let $u, v\in\bbN$. As in~\cite{IofShlosVel2014},
$\calP^{u,v}_{M,N,+}$ is used to denote the set of trajectories 
${\bbX}$ starting at $u$ at time $M$, ending at $v$ at time $N$ and 
staying positive during the time interval $\{M,\ldots,N\}$. 

Let $\uu,\uv\in\bbN^n\cap\bbA_n^+$ and $M,N\in\bbZ$ 
with $M\leq N$.
Let $\calP^{\uu,\uv}_{M,N,+}$ be the family of 
$n$ trajectories $\underline{\bbX}$ starting at $\uu$ at time $M$, 
ending at $\uv$ at time $N$
and satisfying
\be{eq:Pos-constaint1}
0 < X^1_j < X^2_j <\dots < X^n_j\quad \forall j\in \{M+1,\ldots,N-1\} .
\ee

For $\uu,\uv\in\bbN^n$, let 
$\calA^{\uu,\uv}_{M,N,+}$ be the set of $n$ 
trajectories $\underline{\bbX}$ starting at $\uu$ at time $M$, 
ending at $\uv$ at time $N$, staying positive during the time 
interval $\{M, \ldots ,N\}$ and satisfying 
\be{eq:Pos-constaint2}
X^\ell_j \neq  X^k_j \quad \forall j\in \{M,\ldots,N\} \text{ and } 
\ell\neq k.
\ee

For $N>0$, we shall use the shorthand notations
$\calP^{\uu,\uv}_{N,+}
= 
\calP^{\uu,\uv}_{-N, N,+}$ and
$\hat\calP^{\uu ,\uv}_{N,+}
=
\calP^{\uu ,\uv}_{1,N,+}$.
The same convention applies for the shorthand notations  
$\calA^{\uu ,\uv}_{N,+}$ and 
$\hat\calA^{\uu ,\uv}_{N,+}$. 

The model which we define below is a polymer measure over 
ordered trajectories $\calP^{\uu,\uv}_{N,+}$. 
The families $\calA^{\uu,\uv}_{M,N,+}$ are needed
for an application of the Karlin--McGregor formula. 

\subsection{The model}
\label{sub:model}

Let $\{V_\lambda\}_{\lambda>0}$ be a family of self-potentials 
$V_\lambda: \bbN\to\bbR_+$.
For a finite trajectory $\bbX = (X_M,\ldots,X_N)$, let
$\sfp(\bbX) = \prod_{i=M+1}^N \sfp(X_i - X_{i-1})$ be its probability for the 
underlying random walk, and let us introduce the tilted weights  
\be{eq:p-lambda}
\sfw_\lambda (\bbX) = \mathrm{e}^{-\sum_{M+1}^N V_\lambda(X_i)} \sfp(\bbX) . 
\ee 
Given $u,v\in\bbN$ and $\lambda>0$, define the partition functions 
and the probability distributions
\be{eq:pf}
Z^{u,v}_{N,+,\lambda} = 
\sum_{{\bbX}\in\calP^{u,v}_{N,+}} \sfw_\lambda(\bbX)
\text{ and } \bbP^{u,v}_{N,+,\lambda}(\bbX)
= \frac{1}
{Z^{u v}_{N,+,\lambda}}\, \sfw_\lambda(\bbX) 
\1_{\{\bbX\in\calP^{u,v}_{N,+}\} } . 
\ee
In the case of an $n$-tuple $\underline{\bbX} = (\bbX^1,\ldots,\bbX^n)$ of
trajectories, we consider the product weights
$\sfw_\lambda(\underline{\bbX}) = \prod_{i=1}^n \sfw_\lambda(\bbX^i)$.  
If $\calS$ is a finite or countable set of such tuples, then 
the corresponding restricted partition functions are denoted by
\be{eq:pfS}
Z_\lambda \left[\calS\right] = 
\sum_{\underline{\bbX}\in\calS
} 
\sfw_\lambda(\underline{\bbX}) .
\ee
We shall use the shorthand notations 
$Z^{\uu,\uv}_{N,+,\lambda} = 
Z_\lambda 
[\calP^{\uu,\uv}_{N,+}]$ and
$\hat Z^{\uu,\uv}_{N,+,\lambda} = 
Z_\lambda 
[\hat \calP^{\uu,\uv}_{N,+}]$. 
Finally, let us define the probability distribution
$\bbP^{\uu,\uv}_{N,+,\lambda}$
on $\calP^{\uu,\uv}_{N,+}$ by
\be{eq:pd}
\bbP^{\uu,\uv}_{N,+,\lambda} (\underline{\bbX}) 
=
\frac{1}{Z^{\uu,\uv}_{N,+,\lambda}}\,
\sfw_\lambda(\underline{\bbX}) \1_{\{\underline{\bbX} \in 
\calP^{\uu,\uv}_{N,+}\} }  .
\ee
The term $\sum_{-N+1}^N V_\lambda (X_i)$ represents a generalized (non-linear)
area below the trajectory $\bbX$. It reduces to (a multiple of) the usual area 
when $V_\lambda(x) = \lambda x$. As in~\cite{IofShlosVel2014}, we make the 
following set of assumptions on $V_\lambda $:

\subsection{Assumptions on $V_\lambda$ and the scale $H_\lambda$}

For any $\lambda>0$, the function $V_\lambda$ on $[0,\infty)$  is 
continuous, monotone increasing and 
satisfies
\be{eq:VL-1}
V_\lambda (0) = 0 \quad\text{ and }\quad \lim_{x\to\infty} V_\lambda(x) = 
\infty .
\ee
In particular, the relation
\be{eq:HL-1}
\Hla^2 V_\lambda(\Hla) = 1
\ee
determines unambiguously the quantity $\Hla$. Furthermore, we make the 
assumptions that $\lim_{\lambda\downarrow 0} \Hla =\infty$ and that there 
exists a function $q\in\sfC^2(\bbR^+)$ such that 
\be{eq:HL-2}
\lim_{\lambda\downarrow 0} \Hla^2 V_\lambda(r\Hla) = q(r),
\ee
uniformly on compact subsets of $\bbR_+$. 
Note  that $\Hla$, respectively $\Hla^2$, plays the role of the spatial, 
respectively temporal, scale in the invariance principle which is formulated 
below in Theorem~\ref{thm:A}. 

Furthermore, we shall assume that there exist $\lambda_0>0$ and a (continuous 
non-decreasing) function $q_0\geq 0$ with $\lim_{r\to\infty} q_0(r) = \infty$ 
such that, for all $\lambda\leq\lambda_0$,
\be{eq:HL-3}
\Hla^2 V_\lambda (r\Hla) \geq q_0(r) \text{ on } \bbR_+ .
\ee
Finally, we assume that $q_0$ grows to $\infty$ sufficiently fast; namely, for 
any $\kappa>0$, 
\be{eq:Extra-A} 
\int_0^\infty \mathrm{e}^{-\kappa q_0(r) }\,\dd r <\infty .
\ee
Presumably, our main results hold without assumption~\eqref{eq:Extra-A}. 
However, since it is rather soft and since it implies the claim of the 
technically very convenient Lemma~\ref{lem:empt-sum} below, we decided to keep 
it.

\begin{remark} 
A natural class of examples of family of potentials 
satisfying assumptions \eqref{eq:VL-1}-\eqref{eq:Extra-A} is given by 
$V_\lambda (x ) = \lambda x^\alpha$ with $\alpha>0$. For the latter, 
$\Hla=\lambda^{-1/(2+\alpha)}$ and $q(r)=q_0(r)=r^\alpha$. In this way, the 
case of linear area tilts $\alpha =1$ corresponds to the familiar 
Airy rescaling $H_\lambda = \lambda^{-1/3}$. 
\end{remark}

\subsection{The result}
\label{sub:Main}

We set  $h_\lambda = \Hla^{-1}$. 
The paths are rescaled as follows: For $t\in h_\lambda^2\bbZ$, define
\be{eq:xN}
\ux^\lambda (t)
= h_{\lambda} \uX_{\Hla^2 t}
= \frac{1}{H_\lambda}\uX_{\Hla^2 t} . 
\ee
Then, extend $\ux^\lambda$ to any $t\in\bbR$ by linear interpolation. 
In this way, given $T>0$ and $\uu, \uv$,
we can talk about the induced distribution
$\bbP^{\uu,\uv;T}_{N,+,\lambda}$
on the space of continuous functions $\sfC\bigl([-T,T],\bbA_n^+\bigr)$.

\begin{bigthm}
\label{thm:A}
Let $\lambda_N$ be a sequence satisfying 
\be{eq:lambdaN-cond} 
\lim_{N\to\infty}\lambda_N = 0\quad \text{ and }\quad 
\lim_{N\to\infty} a_N \df \lim_{N\to\infty}\frac{N}{H_{\lambda_N}^2} = \infty .
\ee
Fix any $C\in (0,\infty)$ and any $T>0$. 
Then, the sequence of distributions
$\bbP^{\uu,\uv ;T}_{N,+,{\lambda_N}}$ 
converges weakly to the distribution $\bbP_n^{+ ;T}$ of the ergodic 
diffusion $\ux(\cdot)$ in~\eqref{eq:FS-diffusion}, uniformly in $v_n, 
u_n\leq  C H_{\lambda_N}$. 
\end{bigthm}

\subsection{Non-strict constraints} 

In the sequel, we shall focus on the strict constraints expressed 
in~\eqref{eq:Pos-constaint1}. However, a rather straightforward modification of 
our arguments would imply that the conclusion still holds when the ordering 
in~\eqref{eq:Pos-constaint1} is non-strict, that is, when we instead require 
that
\be{eq:Pos-constaint3}
0  \leq  X^1_j \leq  X^2_j \leq \ldots \leq X^n_j\quad \forall j\in 
\{M,\ldots,N\} .
\ee
Namely, let $\calP^{\uu,\uv}_{M,N,0}$ be the family 
of $n$ trajectories $\underline{\bbX}$ starting at $\uu$ at time 
$M$, ending at $\uv$ at time $N$ and 
satisfying~\eqref{eq:Pos-constaint3}. As in the case of strict ordering, we 
use abbreviation $\calP^{\uu,\uv}_{N,0} = 
\calP^{\uu,\uv}_{-N,N,0}$. Define 
(recall~\eqref{eq:RW-measure})
\begin{align*}
Z^{\uu,\uv}_{N,0,\lambda}
&= 
\sum_{\underline{\bbX}\in\calP^{\uu,\uv}_{N,0}}
\mathrm{e}^{-\sum_{\ell=1}^n\sum_{i=-N}^N V_\lambda(X_i^\ell)}
\, \mathbf{P} (\underline{{\bbX}})
\intertext{and}
\bbP^{\uu ,\uv}_{N,0,\lambda} (\underline{\bbX} )
&=
\frac{\mathrm{e}^{-\sum_{\ell=1}^n\sum_{i=-N}^N 
V_\lambda(X_i^\ell)}{\mathbf{P}}(\underline{\bbX})}
{Z^{\uu,\uv}_{N,0,\lambda}}
\, \1_{\{\underline{\bbX}\in\calP^{\uu,\uv}_{N,0}\}} 
. 
\end{align*}
\begin{corollary}
\label{cor:not}
Under the same assumptions, the conclusions of Theorem~\ref{thm:A} hold for 
the family of measures $\bbP^{\uu,\uv}_{N,0,\lambda}$.
\end{corollary}

\subsection{Rescaling and the corresponding notation} 
\label{sub:resc-notation}

It will be  convenient to adjust our notations to the running scales 
$h_\lambda$. Define:
\be{eq:AlNl-def}
\bbN_\lambda =
h_\lambda \bbN, \quad \bbA_{n,\lambda}^+ = \bbA_n^+\cap (\bbN_\lambda)^n 
\quad \text{and}\quad \bbZ_\lambda = h_\lambda^2 \bbZ .
\ee
In this way, $\ux^\lambda (t)$ in~\eqref{eq:xN} belongs to $\bbA_{n,\lambda}^+$ 
for every $t\in\bbZ_\lambda$.

For $a,b,t\in\bbZ_\lambda$ and $\ur,\us\in\bbA_{n,\lambda}^+$, we shall write, 
with a slight abuse of notation,
\be{eq:scaling-lambda}
\calP^{\ur ,\us}_{a,b,+,\lambda} \equiv 
\calP^{\Hla\ur,\Hla\us}_{\Hla^2 a,\Hla^2 b,+,\lambda} \text{ and
similarly for }
\calP^{\ur,\us}_{t,+,\lambda},\hat\calP^{\ur,\us}_{t,+,\lambda}
\calA^{\ur,\us}_{t,+,\lambda} \text{ and } \hat\calA^{\ur,\us}_{t,+,\lambda} .
\ee
The same conventions apply to partition functions (e.g., we shall write 
$\hat Z^{\ur ,\us}_{t,+,\lambda} = 
Z_\lambda [\hat\calP^{\ur,\us}_{t,+,\lambda}]$) and for probability 
distributions (e.g., we shall write $\bbP^{\ur,\us}_{a,+,\lambda}$ for 
$a\in\bbZ_\lambda$ and $\ur,\us,\in\bbA_{n,\lambda}^+$). 

With the above notations, Theorem~\ref{thm:A} can be restated as follows: 
Let
\be{eq:lambdaN-cond-resc}
\lim_{N\to\infty} \lambda_N = 0 \quad \text{and}\quad
\lim_{N\to\infty}a_N = \infty .
\ee
Then, the family of distributions 
$\bbP^{\ur,\us;T}_{a_N,+,{\lambda_N}}$ converges weakly to the distribution 
$\bbP_n^{+ ;T}$ of the ergodic 
diffusion $\ux(\cdot)$, uniformly in $r_n,s_n\leq C$.

\medskip
Our proofs rely on the properties of the underlying rescaled random walks
(without area tilts). 
The corresponding notation for the latter follows the above convention 
adopted for polymer 
measures: 
Given $\lambda>0$ and $\ur\in\bbN_\lambda^n$, we use 
$\hat{\mathbf{P}}_\lambda^{\ur}$ for the law of the rescaled walk started at 
time zero at $\ur$. The \emph{restriction} of $\hat{\mathbf{P}}_\lambda^{\ur}$ 
to the set of trajectories which stay in $\bbA_{n,\lambda}^+$ during the 
interval $[0,t]$ is denoted by $\hat{\mathbf{P}}_{t,+,\lambda}^{\ur}$. 
When the end-point $t$ is clear from the context, we will sometimes use the  
shorthand notation $\hat{\mathbf{P}}_{+,\lambda}^{\ur}$. 
Finally, given $\us\in\bbN_\lambda^n$ and $t\in\bbZ_\lambda$, we use
\be{eq:cond-RW-not}
\hat{\mathbf{P}}_{t,\lambda}^{\ur,\us} = 
\hat{\mathbf{P}}_{\lambda}^{\ur} \bigl(\,\cdot\, \bigm|
\ux^\lambda (t) = \us\bigr) \quad \text{and}\quad 
\hat{\mathbf{P}}_{t,+,\lambda}^{\ur,\us} = 
\hat{\mathbf{P}}_{+,\lambda}^{\ur} \bigl(\,\cdot\, \bigm| 
\ux^\lambda (t) = \us\bigr) .
\ee

\subsection{Structure of the argument}
\label{sub:Structure}
As $\lambda\downarrow 0$, the following notion of convergence is employed: 
Consider the spaces $\ell_2(\bbN_\lambda)$ and $\ell_2(\bbA_{n,\lambda}^+)$ 
with scalar products
\be{eq:s-l-spaces}
\langle f,g\rangle_{2,\lambda} =
h_\lambda \sum_{r\in\bbN_\lambda} f(r)g(r)
\text{ and, respectively, }
\langle f,g\rangle_{2,\lambda} = h_\lambda^n \sum_{\ur\in\bbA_{n,\lambda}^+}
f(\ur)g(\ur) .
\ee
Let $\rho_\lambda : \bbL_2(\bbR_+) \to \ell_2(\bbN_\lambda)$ and 
$\rho_{\lambda,n} : \bbL_2(\bbA_n^+) \to \ell_2(\bbA_{n,\lambda}^+)$ 
be linear contractions; for instance, to fix the ideas, set
\be{eq:rho-maps}
\rho_\lambda u( r) =
\frac1{h_\lambda} \int_{(r-h_\lambda)_+}^r u(s)\,\dd s
\text{ and }
\rho_{\lambda,n} u(\ur) =
\frac{1}{h_\lambda^n} \int_{(r_1-h_\lambda)_+}^{r_1}
\!\!\!\!
\cdots 
\int_{(r_n-h_\lambda)_+}^{r_n} u(\us) \1_{\{\us\in\bbA_n^+\}} \,\dd\us .
\ee
Above, $s_+ = s\vee 0$ for any $s\in\bbR$.
Let us say that a sequence $u_\lambda \in \ell_2(\bbN_\lambda)$ converges to
$u\in\bbL_2(\bbR_+)$, which we denote $u = \lim u_\lambda$, if
\be{eq:L-conv}
\lim_{\lambda\downarrow 0} \| u_\lambda - \rho_\lambda u\|_{2,\lambda} = 0 .
\ee
The same definition applies for sequences 
$u_\lambda\in\ell_2(\bbA_{n,\lambda}^+)$ and, accordingly, for the limiting $u$ 
in $\bbL_2(\bbA_n^+)$. Note that, in both cases, if
$u = \lim_{\lambda\downarrow 0}u_\lambda$ and 
$v = \lim_{\lambda\downarrow 0}v_\lambda$, then
\be{eq:sp-conv}
\lim_{\lambda\downarrow 0} \langle u_\lambda,v_\lambda \rangle_{2,\lambda} = 
\int_0^\infty u(r)v (r)\,\dd r
\text{, respectively, }
\lim_{\lambda\downarrow 0}\langle u_\lambda,v_\lambda \rangle_{2,\lambda} = 
\int_{\bbA_n^+} u(\ur)v(\ur)\,\dd\ur . 
\ee

\smallskip
\noindent
\step{1} (Convergence of one-dimensional and product semi-groups.) 
Recall that $\sfT_t$ is an integral operator whose kernel $h_t$ is
defined in~\eqref{eq:ht}. \cite[Proposition~3]{IofShlosVel2014} implies that if 
a sequence $f_\lambda \in\ell_2 (\bbN_\lambda)$ converges to $f\in\bbL_2 
(\bbR_+)$, then, for any $t>0$,
\be{eq:1-conv}
\lim_{\lambda\downarrow 0} \sum_{s\in\bbN_\lambda}
\hat Z_{t,+,\lambda}^{r,s} f_\lambda (s) =
\int_0^\infty  h_t(r,s) f(s)\,\dd s ,
\ee
in the sense of~\eqref{eq:L-conv} above. In particular, for any 
$f,g\in\sfC_0(\bbR_+)$, 
\be{eq:fg-one}
\lim_{\lambda\downarrow 0} h_\lambda\sum_{r\in\bbN_\lambda} 
\sum_{s\in\bbN_\lambda} g(r) \hat Z_{t,+,\lambda}^{r,s} f(s) = 
\int_0^\infty \int_0^\infty g(r) h_t(r,s)f(s)\,\dd r\dd s .
\ee
We claim:
\begin{proposition}
\label{prop:conv-n}
Assume that the sequence $f_\lambda\in\ell_2(\bbA_{n,\lambda}^+)$ converges
to $f\in
\sfC_0 (\bbA_n^+)$. 
Let $\sigma$ be a permutation of $\{1,\ldots,n\}$.
Then, for any $t>0$,
\be{eq:conv-n} 
\lim_{\lambda\downarrow 0}
\sum_{\us\in\bbA_{n,\lambda}^+} \prod_{i=1}^n \hat 
Z_{t,+,\lambda}^{r_i,s_{\sigma_i}} f_\lambda (\us) =
\int_{\bbA_n^+} \prod_{i=1}^n h_t(r_i,s_{\sigma_i}) f(\us) \,\dd\us , 
\ee
in the sense of~\eqref{eq:L-conv} above.
In particular, let $f,g\in\sfC_0(\bbA_n^+)$. Then, for any $t>0$, 
\be{eq:conv-n-gf}
\lim_{\lambda\downarrow 0} 
h_\lambda^n \sum_{\ur\in\bbA_{n,\lambda}^+}
\sum_{\us\in\bbA_{n,\lambda}^+}
g(\ur) \prod_{i=1}^n \hat Z_{t,+,\lambda}^{r_i,s_{\sigma_i}} f(\us) =
\int_{\bbA_n^+} \int_{\bbA_n^+} g(\ur) \prod_{i=1}^n h_t(r_i,s_{\sigma_i}) 
f(\us) \,\dd\ur \dd\us .
\ee
\end{proposition}

\smallskip
\noindent
\step{2} (Karlin--McGregor formula and probabilistic estimates.)
Let $\ur,\us\in\bbA_{n,\lambda}^+$. By an application of Karlin--McGregor 
formula (see~\cite[Section~5]{KMcG}), 
\be{eq:KM-lambda}
\operatorname{det}\bigl\{ \hat Z_{t,+,\lambda}^{r_i,s_j}\bigr\} = 
\sum_\sigma (-1)^{\operatorname{sgn}(\sigma)}
Z[\hat\calA_{t,+,\lambda}^{\ur,\us_\sigma}] .
\ee
Above, $(\us_\sigma)_i \equiv s_{\sigma_i}$. 

\smallskip 
Recall our notation for rescaled norms: 
$
\|f_\lambda\|_{2,\lambda}^2 = h_\lambda^n \sum_{\ur} f_\lambda^2(\ur) 
$.
We claim:
\begin{theorem}
\label{prop:CM-Prob}
\textbf{(a)} For any $t_0>0$ and for any non-trivial permutation
$\sigma\neq\mathsf{Id}$, 
\be{eq:sigma-triv}
\lim_{\lambda\downarrow 0}
\sum_{\us\in\bbA_{n,\lambda}^+} 
Z_\lambda[\hat\calA_{t,+,\lambda}^{\ur,\us_\sigma}] f_\lambda(\us) = 0 ,
\ee
in the sense of~\eqref{eq:L-conv}, uniformly in $t\geq t_0$ and in 
$\|f_\lambda\|_{2,\lambda} = 1$.
\\
\textbf{(b)} For any $t_0>0$,
\be{eq:Id-dif}
\lim_{\lambda\downarrow 0} \sum_{\us\in\bbA_{n,\lambda}^+}
\bigl( 
Z_\lambda[\hat\calA_{t,+,\lambda}^{\ur,\us}]
- 
Z_\lambda[\hat\calP_{t,+,\lambda}^{\ur,\us}]
\bigr) 
f_\lambda(\us) 
= 0 ,
\ee
as well, also uniformly in $t\geq t_0$ and $\|f_\lambda\|_{2,\lambda} = 1$. 
\end{theorem}

Recall our notation 
$\hat Z_{t,+,\lambda}^{\ur,\us} = Z_\lambda[\hat\calP_{t,+,\lambda}^{\ur,\us}]$ 
and 
$\kappa_t(\ur,\us) = \operatorname{det}\{h_t(r_i,s_j)\}$. 
Proposition~\ref{prop:conv-n} and Theorem~\ref{prop:CM-Prob} imply:
\begin{theorem}
\label{thm:KMcG}
For any $t>0$ and any sequence $f_\lambda\in\ell_2(\bbA_{n,\lambda}^+)$ with 
$\lim_{\lambda\downarrow 0} f_\lambda = f$, 
\be{eq:KM-conv}
\lim_{\lambda\downarrow 0} \sum_{\us\in\bbA_{n,\lambda}^+} 
\hat Z_{t,+,\lambda}^{\ur,\us} f_\lambda(\us) =
\int_{\bbA_n^+} \kappa_t(\ur,\us) f(\us) \,\dd\us . 
\ee
In particular, for any $f,g\in\sfC_0(\bbA_n^+)$,
\be{eq:KM-conv-fg}
\lim_{\lambda\downarrow 0} h_\lambda^n 
\sum_{\ur\in\bbA_{n,\lambda}^+} \sum_{\us\in\bbA_{n,\lambda}^+}
g(\ur) Z_{t,+,\lambda}^{\ur,\us} f(\us) = 
\int_{\bbA_n^+} \int_{\bbA_n^+} g(\ur) \kappa_t(\ur,\us) f(\us) \,\dd\ur\dd\us 
.
\ee
\end{theorem}

\noindent
\step{3} (Tightness.)
We claim: 
\begin{proposition}
\label{prop:tight}
Fix any $T>0$. Under the conditions of Theorem~\ref{thm:A}, the family
$\{\bbP_{a_N,+,\lambda_N}^{\ur,\us;T}\}$ of probability distributions on 
$\sfC([-T,T],\bbA_n^+)$ is tight.
\end{proposition}
 
\noindent
\step{4} (Mixing.)
We claim:
\begin{theorem}
\label{prop:mix}
For any $C<\infty$, there exist $c_1,c_2>0$ such that, for any $K>0$,
\be{eq:mix}
\bigl\| \bbP_{a,+,\lambda}^{\ur,\us;T} - 
\bbP_{b,+,\lambda}^{\uw,\uz;T} \bigr\|_{\mathrm{var}}
\leq c_1 \mathrm{e}^{-c_2 K} , 
\ee
holds uniformly in $\lambda$ small, $a,b\in\bbZ_\lambda$ with $a,b\geq (K+T)$
and uniformly in $\ur,\us,\uw,\uz\in\bbA_{n,\lambda}^+$ with 
$r_n,s_n,w_n,z_n\leq C$.
\end{theorem}

\noindent
\step{5} (Convergence of finite-dimensional distributions.)
Fix $T >0$. Let $\lambda_N\downarrow 0$ and let $a_N\in\bbZ_{\lambda_N}$ 
satisfy 
$\lim a_N = \infty$.  
Let $f,g\in\sfC_0(\bbA_n^+)$ be two non-negative and non-identically zero 
functions.
For $M\in\bbZ_{\lambda}$, $M>T$, define the partition functions (rescaled as 
in~\eqref{eq:scaling-lambda})
\be{eq:Zgf}
Z_{M,+,\lambda}^{g,f} = h_\lambda^n \sum_{\ur\in\bbA_n^+} \sum_{\us\in\bbA_n^+}
g(\ur) Z_{M,+,\lambda}^{\ur,\us} f(\us) >0
\ee
and
let $\bbP_{M,+,\lambda}^{g,f;T}$ be the corresponding induced probability 
distribution on $\sfC([-T,T],\bbA_n^+)$. By Theorem~\ref{prop:mix}, under the
conditions of Theorem~\ref{thm:A},
\be{eq:Klim-var}
\lim_{M\to\infty} \lim_{N\to\infty}
\bigl\| \bbP_{a_N,+,\lambda_N }^{\ur,\us;T}
- 
\bbP_{M,+,\lambda_N }^{g,f;T} \bigr\|_{\mathrm{var}} = 0 ,
\ee
uniformly in $r_n,s_n\leq C$. 

Let now $-T\leq t_1 < t_2 <\ldots t_m \leq T$ and let 
$u_1,\ldots,u_m \in \sfC_0(\bbA_n^+)$.
By Theorem~\ref{thm:KMcG},
\begin{multline}
\lim_{\lambda\downarrow 0} \bbE_{M,+,\lambda}^{g,f;T} \Bigl(\prod_{i=1}^m 
u_i\bigl(\ux^\lambda(t_i)\bigr) \Bigr) = \\
\frac{ 
\int g(\ur) \int \kappa_{t_1+M}(\ur,\ur^1) u_1(\ur^1) 
\int\ldots u_m(\ur^m) \int \kappa_{M-t_m}(\ur^m,\us) f(\us) 
\,\dd\us\dd\ur^m\cdots\dd\ur
}
{
\int\int g(\ur) \kappa_{2M}(\ur,\us) f(\us) \,\dd\us\dd\ur
} . 
\label{eq:gf-lim1}
\end{multline}
Above, all integrals are over $\bbA_n^+$. In view of~\eqref{eq:Vdm} and by the 
definition of the semigroup $\sfS_t^+$ in~\eqref{eq:Vdm-lim}, the 
formulas~\eqref{eq:Klim-var} and~\eqref{eq:gf-lim1} imply: 
\begin{proposition}
\label{prop:fdd}
Fix $T>0$, $-T \leq t_1 < t_2 <\ldots t_m \leq T$ and let $u_1,\ldots,u_m$ be
bounded continuous functions on $\bbA_n^+$. Let $\lambda_N$ and $a_N$ satisfy 
the assumptions of Theorem~\ref{thm:A}. Then,
\begin{multline}
\lim_{N\to\infty} \bbE_{a_N,+,\lambda_N}^{\ur,\us} \Bigl( \prod_{i=1}^m 
u_i\bigl(\ux^{\lambda_N}(t_i)\bigr) \Bigr) = \\
\int \!\! \Delta^2(\ur^1) u_1(\ur^1) \! \int\!\! q_{t_2-t_1}(\ur^1,\ur^2) 
u_2(\ur^2) \! \int \! \cdots \!
\int \!\! q_{t_m-t_{m-1}}(\ur^{m-1},\ur^m) \,\dd\ur^m \cdots\dd\ur^1 ,
\label{eq:n-int}
\end{multline}
uniformly in $r_n,s_n\leq C$. 
Above, $q_t$ is the transition kernel of $\sfS_t^+$, as defined 
in~\eqref{eq:qt-kernel}.
\end{proposition}

\noindent
\step{6} (Conclusion of the Proof.) 
By Proposition~\ref{prop:tight}, the sequence of measures 
$\bigl\{\bbP_{a_N,+,\lambda_N}^{\ur,\us;T};\, {r_n,s_n\leq C}\bigr\}$ on 
$\sfC([-T,T],\bbA_n^+)$ is tight for any $T>0$ fixed. By 
Proposition~\ref{prop:fdd}, its finite-dimensional distributions converge to 
the finite-dimensional distributions of the Dyson diffusion $\ux(\cdot)$ 
in~\eqref{eq:FS-diffusion}.
\qed 

\subsection{Organization of the technical part of the paper.} 
\label{sub:Org}
We still have to prove Theorem~\ref{prop:CM-Prob}, 
Propositions~\ref{prop:conv-n} and~\ref{prop:tight} and Theorem~\ref{prop:mix}. 
This will be done in Section~\ref{sec:CM-Prob}, Section~\ref{sec:props}
and, respectively, in Sections~\ref{sec:mix} and~\ref{sec:lems}. 
The proof of Theorem~\ref{prop:mix} 
 is by far the most techically loaded part of the paper,  and it 
relies on the probabilistic 
estimates~\eqref{eq:Inp1}-\eqref{eq:Inp3}, which are based on strong 
approximation techniques and on invariance principles for random walks in Weyl 
chambers. The derivation of~\eqref{eq:Inp1}-\eqref{eq:Inp3} is relegated to the 
Appendix.

\section{Proof of Theorem~\ref{prop:CM-Prob}}
\label{sec:CM-Prob} 

\subsection{Preliminary estimates.} 

Let us start with three preliminary estimates.
The first one is just a rough local CLT estimate for the underlying random 
walk without 
area tilts: Recall that whenever we write quantities like 
$\hat Z_{t,+,\lambda}^{r,s}$, we are implicitly assuming that 
$t\in\bbZ_\lambda =h_\lambda^2\bbZ$ and that $r,s\in \bbN_\lambda = h_\lambda \bbN$.
\begin{lemma}
\label{lem:CLT}
For any $t_0>0$, there exists a finite constant $c_1(t_0)$ such that
\be{eq:CLT}
\sup_{t\geq t_0} \sup_{r,s\in\bbN_\lambda} \hat Z_{t,+,\lambda}^{r,s}
\leq c_1(t_0) h_\lambda, 
\ee 
for all $\lambda$ sufficiently small.
\end{lemma}
\noindent
Indeed, since $V_\lambda \geq 0$, $Z_{t,+,\lambda}^{r,s}\leq 
\hat{\mathbf P}_\lambda^r\lb x^\lambda (t) = s\rb$. \qed
\smallskip 

\noindent 
Next, following~\cite{IofShlosVel2014}, let us introduce
\be{eq:Zempty}
\hat Z_{t,+,\lambda}^{r,\emptyset}
=
\sum_{s\in\bbN_\lambda} \hat Z_{t,+,\lambda}^{r,s} .
\ee
\begin{lemma}
\label{lem:empt-sum}
For any $t_0>0$, there exists a finite constant $c_2(t_0)$ such that
\be{eq:empt-sum}
\sup_{t \geq t_0} h_\lambda \sum_{r\in\bbN_\lambda} \hat 
Z_{t,+,\lambda}^{r,\emptyset} \leq c_2(t_0) ,
\ee
for all $\lambda$ sufficiently small. Furthermore, 
\be{eq:K-lim}
\lim_{K\to\infty} 
\sup_{t \geq t_0 } h_\lambda \sumtwo{r\in\bbN_\lambda}{r\geq K} \hat 
Z_{t,+,\lambda}^{r,\emptyset} = 0 ,
\ee
uniformly in $\lambda$ sufficiently small.
\end{lemma}
\begin{proof}
Note that Lemma~\ref{lem:empt-sum} is in general  wrong without the additional  
Assumption~\eqref{eq:Extra-A}.
On the other hand, under Assumption~\eqref{eq:HL-3}, it is straightforward 
to 
check that there exists $\kappa = \kappa(t_0) > 0$ such that
\be{eq:kappa-bound}
\sup_{t \geq t_0} \hat Z_{t,+,\lambda}^{r,\emptyset} \leq
\mathrm{e}^{ -\kappa \min\{q_0(r/2),r^2\} } ,
\ee
for all $\lambda$ small and all $r\in\bbN_\lambda$. Both~\eqref{eq:empt-sum} 
and~\eqref{eq:K-lim} follow now from \eqref{eq:Extra-A}. 
\end{proof}

The third estimate is again on the underlying random walk, or more precisely on 
two independent copies 
$\lb x^\lambda , y^\lambda \rb$ 
of this walk. 
Namely, 
\begin{lemma} 
\label{lem:two-walks-flip}
For any $\delta_0\in\bbR_+$ and $K\in \bbR_+$ fixed,  
\be{eq:two-walks-flip}
\lim_{\lambda \to 0} \max_{\delta\geq \delta_0} \max_{0 < u <  v \leq K} 
\hat{\mathbf P}^u_\lambda  \otimes \hat{\mathbf P}^v_\lambda  \bigl(
x^\lambda (\delta ) > y^\lambda (\delta )\, ; \, x^\lambda (t )\neq y^\lambda (t ) \, 
\forall\,  t\in [0,\delta]\cap\bbZ_\lambda   \bigr) = 0.
\ee
\end{lemma} 
\begin{proof} 
The claim  follows from~\cite[Theorem~1]{Uchi2011} and local limit asymptotics 
for random walks with exponential tails.
\end{proof}

\subsection{Proof of Theorem~\ref{prop:CM-Prob} (a).}    

Pick $f_\lambda\in\ell_2(\bbA_{n,\lambda}^+)$ with $\|f_\lambda\|_{2,\lambda}^2 
= 1$.
Set  
\be{eq:u-lambda}
u_\lambda (\ur) = \sum_{\us\in\bbA_{n,\lambda}^+} 
Z_\lambda[\hat\calA_{t,+,\lambda}^{\ur,\us_\sigma}] f_\lambda(\us) . 
\ee
In order to prove~\eqref{eq:sigma-triv}, we need to check that, whenever 
$\sigma\neq\mathsf{Id}$,
\be{eq:u-lambda-0}
\lim_{\lambda\downarrow 0} h_\lambda^n \sum_{\ur\in\bbA_{n,\lambda}^+} 
u_\lambda(\ur)^2 = 0 ,
\ee
uniformly in $f$ with $\|f_\lambda\|_{2,\lambda}^2 = 1$. 
By the Cauchy--Schwarz inequality and Lemma~\ref{lem:CLT}, 
\begin{align}
u_\lambda (\ur)^2 = \Bigl( \sum_{\us\in\bbA_{n,\lambda}^+} 
Z_\lambda[\hat\calA_{t,+,\lambda}^{\ur,\us_\sigma}] f_\lambda(\us) \Bigr)^2 
&\leq 
\Bigl( \sum_{\us\in\bbA_{n,\lambda}^+} 
Z[\hat\calA_{t,+,\lambda}^{\ur,\us_\sigma}] \Bigr) 
\Bigl( \sum_{\us\in\bbA_{n,\lambda}^+} 
Z_\lambda[\hat\calA_{t,+,\lambda}^{\ur,\us_\sigma}] f_\lambda^2(\us) \Bigr) 
\notag\\ 
&\leq
c_1 (t_0)^n \sum_{\us\in\bbA_{n,\lambda}^+}
Z_\lambda[\hat\calA_{t,+,\lambda}^{\ur,\us_\sigma}] .
\label{eq:CSch}
\end{align}
If $\sigma\neq\mathsf{Id}$, then there exist $i<j$ such that
$\sigma_i > \sigma_j$. In this case,
\[
\sum_{\us\in\bbA_{n,\lambda}^+} 
Z_\lambda[\hat\calA_{t,+,\lambda}^{\ur,\us_\sigma}]
\leq
\chi^+_{t,\lambda} (r_i,r_j) \prod_{k\neq i,j} Z_{t,+,\lambda}^{r_k,\emptyset} 
, 
\]
where $Z_{t,+,\lambda}^{r,\emptyset}$ were defined in~\eqref{eq:Zempty} and,  
for $\ur,\us\in\bbA_{2,\lambda}^+$, we define $\us^* = (s_1,s_2)^* = (s_2,s_1)$ 
and 
\be{eq:chi-plus}
\chi^+_{t,\lambda} (\ur) = \sum_{\us\in\bbA_{2,\lambda}^+}
Z_\lambda[\hat\calA_{t,+,\lambda}^{\ur,\us^*}] .
\ee
In view of Lemma~\ref{lem:empt-sum}, \eqref{eq:u-lambda-0} would follow once we 
check that 
\be{eq:A-tocheck}
\lim_{\lambda\downarrow 0} h_\lambda^2 \sum_{\ur\in\bbA_{2,\lambda}^+} 
\chi^+_{t,\lambda}(\ur) = 0 .
\ee
Given $K>0$, let us define
\be{eq:eps-K}
\epsilon_\lambda (K) = 
\sup_t h_\lambda \sumtwo{r\in\bbN_\lambda}{r\geq K} \hat 
Z_{t,+,\lambda}^{r,\emptyset} .
\ee
By~\eqref{eq:empt-sum} of Lemma~\ref{lem:empt-sum}, 
\be{eq:chi-bound}
h_\lambda^2 \sum_{\ur\in\bbA_{2,\lambda}^+} \chi^+_{t,\lambda}(\ur) \leq 
\epsilon_\lambda^2(K) + 2\epsilon_\lambda(K) c_2(t_0) + 
h_\lambda^2 \sumtwo{\ur\in\bbA_{2,\lambda}^+}{\mathclap{0 < r_1 < r_2 \leq K}} 
\chi^+_{t,\lambda}(\ur) .
\ee
Next, by~\eqref{eq:K-lim} of Lemma~\ref{lem:empt-sum}, the term 
$\epsilon_\lambda(K) \to 0$ as $K$ tends to infinity, uniformly in $\lambda$ 
small enough. Moreover, choosing 
$\delta = \frac{t}{K^2} 
\geq \frac{t_0}{K^2} \df \delta_0$, we infer from 
Lemma~\ref{lem:two-walks-flip} 
that
\be{eq:max-bound}
\lim_{\lambda\downarrow 0} \sup_{t\geq t_0} \max_{0 < r_1 < r_2 \leq K} 
\chi^+_{t,\lambda}(\ur) = 0, 
\ee
and hence the third term in~\eqref{eq:chi-bound} tends to zero (as $\lambda$ 
tends to $0$) for any $K$ fixed. \eqref{eq:A-tocheck} follows. 
\qed

\subsection{Proof of Theorem~\ref{prop:CM-Prob} (b).}   

Fix $t_0>0$. We should check that
\be{eq:Id-dif-1}
\lim_{\lambda\downarrow 0} 
\sup_{\|g_\lambda\|_{2,\lambda} = \|f_\lambda\|_{2,\lambda} = 1} 
h_\lambda^n \sum_{\ur,\us\in\bbA_{n,\lambda}^+}
g_\lambda(\ur)
Z_\lambda[\hat\calA_{t,+,\lambda}^{\ur,\us} 
\setminus \hat\calP_{t,+,\lambda}^{\ur,\us}]
f_\lambda(\us)
= 0, 
\ee
uniformly in $t>t_0$. 
By definition, any path $\ux^\lambda\in\hat\calA_{t,+,\lambda}^{\ur,\us} 
\setminus \hat\calP_{t,+,\lambda}^{\ur,\us}$ 
(of the random walk in discrete $\bbZ_\lambda$-time, rescaled as 
in~\eqref{eq:xN})
has to exit $\bbA_{n,\lambda}^+$ on its way from $\ur$ to $\us$. 
Let $\tau_-$ and $\tau_+$, $\tau_\pm \in \bbZ_\lambda$,  be, respectively, the 
times of the first and the last visits to $\{\bbA_{n,\lambda }^+\}^c$.
Again, by definition of $\hat\calA_{t,+,\lambda}^{\ur,\us}$, the points 
$\ux^\lambda(\tau_{\pm})$ belong to
\[
[\bbA_{n,\lambda}^+]_{\sigma_{\pm}} = \{ \underline{\sfw}_{\sigma_\pm} \,:\, 
\underline{\sfw} \in \bbA_{n,\lambda}^+ \} ,
\]
for some permutation $\sigma_{\pm}\neq\mathrm{Id}$, which of course depends 
on the particular realization of $\ux^\lambda$.  Since either 
\[
\text{(i) } \tau_- \leq t/2 \quad\text{or}\quad
\text{(ii) } t - \tau_+ \leq t/2 ,
\]
\eqref{eq:Id-dif-1} follows by the same arguments as employed for the proof 
of~\eqref{eq:sigma-triv} (although, in case (ii), the latter should be applied 
to the reversed walk).
 
Indeed, let us fix a permutation $\sigma\neq\mathrm{Id}$. Consider the 
following modification of~\eqref{eq:u-lambda}: Set $\eta = \sigma^{-1}$ and 
\be{eq:u-lambda1}
u_\lambda (\ur) = \max_{u\in [t/2,t]} \sum_{\us\in\bbA_{n,\lambda}^+}
Z_\lambda[\hat\calA_{u,+,\lambda}^{\ur,\us_\eta}] f_\lambda(\us) . 
\ee 
For $\ur,\us\in\bbA_{n,\lambda}^+$ define, 
\[
\rho_{\lambda} (\ur,\us) =
Z_\lambda\bigl[ \ux^\lambda(0) = \ur,\, \ux^\lambda(\tau_-) = \us_\sigma \bigr] 
.
\]
Above, $\{ \ux^\lambda(0) = \ur,\, \ux^\lambda(\tau_-) = \us_\sigma \}$ is the 
set of trajectories started at time zero in $\ur$ and arriving, at their first 
exit from $\bbA_{n,\lambda}^+$, to the point $\us_\sigma\in 
[\bbA_{n,\lambda}^+]_\sigma$.
Clearly, for any $\lambda\geq 0$, 
\be{eq:rho-1}
\sum_{\us} \rho_{\lambda}(\ur,\us) \leq 1. 
\ee
Furthermore, under Assumption~\eqref{eq:Extra-A}, there exists a constant $c_3$ 
such that
\be{eq:rho-2} 
\sum_{\ur} \rho_{\lambda}(\ur,\us) \leq c_3, 
\ee
for all $\lambda$ small enough. 

We can now bound from above the contribution of (i) with 
$
x^\lambda (\tau_-)\in [\bbA_{n,\lambda}^+]_\sigma$ to the sum 
in~\eqref{eq:Id-dif-1}: applying the Cauchy--Schwarz inequality and the 
bounds~\eqref{eq:rho-1} and~\eqref{eq:rho-2},
\[
h_\lambda^n \sum_{\ur,\us}
g_{\lambda}(\ur) \rho_{\lambda}(\ur,\us) u_\lambda(\us)
\leq
\sqrt{c_3} \|g_\lambda\|_{2,\lambda} \|u_\lambda\|_{2,\lambda}
\]
and one can proceed as in the Proof of Theorem~\ref{prop:CM-Prob} (a) to show 
that $\lim_{\lambda\downarrow 0} \|u_\lambda\|_{2,\lambda} = 0$, uniformly in 
$f$ such that $\|f_\lambda\|_{2,\lambda}\leq 1$.
\qed

\section{Proof of Propositions~\ref{prop:conv-n} and~\ref{prop:tight} } 
\label{sec:props}

\subsection{Proof of Proposition~\ref{prop:conv-n}} 

Let us consider $f\in\sfC_0(\bbR_+^n)$ and 
$f_\lambda\in\ell_{2,\lambda}(\bbN_\lambda^n)$.
The convergence $\lim_{\lambda\downarrow 0} f_\lambda = f$ is still defined 
via~\eqref{eq:L-conv} with 
\be{eq:rho-mapsFull}  
\rho_{\lambda,n} u(\ur) = \frac{1}{h_\lambda^n} \int_{(r_1-h_\lambda)_+}^{r_1}
\hspace*{-7mm} \cdots \hspace{2mm} \int_{(r_n-h_\lambda)_+}^{r_n} u(\us)
\,\dd\us .
\ee
If $u_\lambda\in\ell_2(\bbN_\lambda^n)$ converges to $u\in\bbL_2(\bbR_+^n)$, 
then, evidently, $\tilde u_\lambda \df u_\lambda\1_{\Anl}\!\!\!\in\ell_2(\Anl)$
converges to $\tilde u\df u\1_{\An}\!\in\bbL_2(\An)$. 
Hence, \eqref{eq:conv-n} will follow if we check that 
\be{eq:conv-n-Full} 
\lim_{\lambda\downarrow 0} \sum_{\us\in\bbN_\lambda^n}
\prod_{i=1}^n \hat Z_{t,+,\lambda}^{r_i,s_i} f_\lambda(\us)
=
\int_{\bbR_+^n} \prod_{i=1}^n h_t(r_i,s_i) f(\us) \dd\us ,
\ee 
whenever $f\in\sfC_0(\bbR_+^n)$ and $\lim_{\lambda\downarrow 0} f_\lambda = 
f$. 
 
Next, we may assume without loss of generality that $f_\lambda = 
\rho_{\lambda,n}f$. Hence, there exists $R>0$ such that both $f$ and $f_\lambda$
vanish for $r\not\in [0,R]^n \df \Omega_R$. In other words, we can restrict 
our attention to $f_\lambda\in\ell_{2,\lambda}(\bbN_\lambda^n\cap\Omega_R)$ and 
$f\in\sfC_0(\Omega_R)$. 
 
The rest is a monotone class argument based on~\eqref{eq:1-conv}: 
Let $\calH_R$ be the family of bounded measurable functions on $\Omega_R$
such that~\eqref{eq:conv-n-Full} holds. By~\eqref{eq:1-conv}, the family 
$\calH_R$ contains all the products $\prod_{i=1}^n f_i(r_i)$ of bounded and 
measurable functions $f_1,\ldots,f_n$ on $[0,R]$. In particular, 
$\1_{\Omega_R}\in\calH_R$. Next, by linearity, $f,g\in\calH_R$ clearly implies 
that $af+bg\in\calH_R$ for any $a,b\in\bbR$. Finally, if 
\[
0 \leq f^{(1)} \leq f^{(2)} \leq \ldots 
\]
is a non-decreasing family of functions from $\calH_R$ and if $f=\lim f^{(k)}$
exists and is bounded, then $\lim_{k\to\infty} \|f - f^{(k)}\|_2 = 0$. 
Since $\rho_{\lambda,n}$ are contractions,
$\|f_\lambda - f^{(k)}_\lambda\|_{2,\lambda} \leq \|f - f^{(k)}\|_2$ for all 
$\lambda>0$.
On the one hand, in view of Lemma~\ref{lem:CLT} and~\eqref{lem:empt-sum}, an 
application of the Cauchy--Schwarz inequality yields
\begin{align}
h_\lambda^n \sum_{\ur} \Bigl( \sum_{\us} \prod_{i=1}^n
\hat Z_{t,+,\lambda}^{r_i,s_i} u_\lambda(\us) \Bigr)^2 
&\leq
\Bigl( h_\lambda^n \sum_{\ur,\us} \prod_{i=1}^n \hat 
Z_{t,+,\lambda}^{r_i,s_i} \Bigr) \Bigl( c_1(t_0)^n h_\lambda^n \sum_{\us} 
u_\lambda(\us)^2 \Bigr) \notag\\ 
&\leq
\bigl(c_1(t_0) c_2(t_0)\bigr)^n \|u_\lambda\|_{2,\lambda}^2 ,
\label{eq:u-l-bound}
\end{align}
uniformly in $t\geq t_0$. In particular, 
\[
\bigl\| 
\sum_{\us} \prod_{i=1}^n \hat Z_{t,+,\lambda}^{r_i,s_i}
(f_\lambda(\us) - f^{(k)}_\lambda(\us))
\bigr\|_{2,\lambda}
\leq
\bigl(c_1(t_0) c_2 (t_0)\bigr)^n \|f - f^{(k)}\|_2 ,
\]
uniformly in $t\geq t_0$. On the other hand,
\[
\bigl\|
\int_{\bbR_+^n} \prod_{i=1}^n h_t(r_i,s_i) \bigl( f(\us) - f^{(k)}(\us) \bigr) 
\,\dd\us \bigr\|_2
\leq
\|f - f^{(k)}\|_2 .
\]
\eqref{eq:conv-n-Full} follows, for instance, by a diagonal procedure.  
 
\subsection{Proof of Proposition~\ref{prop:tight}}

Note that our proof of Theorem~\ref{prop:mix} below, and hence our proof of 
Proposition~\ref{prop:fdd}, does not rely on the tightness property which we 
are trying to establish here.
By Proposition~\ref{prop:fdd} the one-dimensional projections of
$\mathbb{P}^{\ur,\us;T}_{a_N,+,\lambda_N}$, that is, the
distributions of $\ux^{\lambda_N}(t)$ for each fixed $\abs{t}\leq T$,
converge.

Then, according to~\cite[Theorem~8.3]{Billingsley},
the family $\{\mathbb{P}^{\ur,\us;T}_{a_N,+,\lambda_N}\}$
is tight if for all positive $\gamma$ and $\beta$ there exist $\delta\in(0,1)$ 
and 
$N_0$ such that, uniformly in $t\in[-T,T]$,
\begin{equation}
\label{tightness.1}
\mathbb{P}^{\ur,\us;T}_{a_N,+,\lambda_N}
\bigl( \sup_{s\in[t,t+\delta]}
|\ux^{\lambda_N}(s) - \ux^{\lambda_N}(t)|
\geq \gamma \bigr)
\leq
\delta\beta, \quad N\geq N_0 .
\end{equation}
Since $a_N$ tends to infinity, it suffices to prove~\eqref{tightness.1} for
$t=0$.

Recall that $T$ is fixed. We may assume that $a_N\gg T$. Now, the exponential 
mixing bound  
\eqref{eq:mix} implies that the following holds uniformly in $M\geq 2T$:  
\begin{multline}
\mathbb{P}^{\ur,\us;T}_{a_N,+,\lambda_N}
\bigl( \sup_{s\in[0,\delta]}
|\ux^{\lambda_N}(s) - \ux^{\lambda_N}(0)|
\geq \gamma \bigr) \\
\leq
e^{-c_1 M} +
\mathbb{P}^{g,f;T}_{M,+,\lambda_N} \bigl( \sup_{s\in[0,\delta]}
|\ux^{\lambda_N}(s) - \ux^{\lambda_N}(0)|
\geq \gamma \bigr) .
\label{tightness.2}
\end{multline}
Since the potentials $V_{\lambda_N}$ in the definition of tilted measures 
are non-negative, the latter probability is controlled in terms of the 
underlying random walk:
\[
\mathbb{P}^{g,f;T}_{M,+,\lambda_N} \bigl( \sup_{s\in[0,\delta]}
|\ux^{\lambda_N}(s) - \ux^{\lambda_N}(0)|
\geq \gamma \bigr)
\leq
\frac{1}{Z^{g,f}_{M,+,\lambda_N}}
\mathbf{P}_{\lambda_N}\bigl( \sup_{s\in[0,\delta]}
|\ux^{\lambda_N}(s) - \ux^{\lambda_N}(0)|
\geq \gamma \bigr) .
\]
It follows from Theorem~\ref{thm:KMcG} and the definition of the kernel 
$\kappa_t$ that there exists $c_2=c_2(g,f)$ such that
\[
Z^{g,f}_{M,+,\lambda_N} \geq e^{-c_2 M} .
\]
From these estimates and~\eqref{tightness.2}, we conclude that
\begin{multline*}
\mathbb{P}^{\ur,\us;T}_{a_N,+,\lambda_N}
\bigl( \sup_{s\in[0,\delta]}
|\ux^{\lambda_N}(s) - \ux^{\lambda_N}(0)|
\geq \gamma \bigr) \\
\leq
e^{-c_1 M} + e^{c_2 M} \mathbf{P}_{\lambda_N} \bigl( \sup_{s\in[0,\delta]}
|\ux^{\lambda_N}(s) - \ux^{\lambda_N}(0)|
\geq \gamma \bigr) .
\end{multline*}
An application of the standard functional CLT (recall our assumption 
\eqref{eq:sigma} on unit variance) 
yields the inequality
\[
\limsup_{N\to\infty} \mathbf{P}_{\lambda_N} \bigl( \sup_{s\in[0,\delta]}
|\ux^{\lambda_N}(s) - \ux^{\lambda_N}(0)|
\geq \gamma \bigr)
\leq
e^{-\gamma^2/4 \delta} .
\]
Consequently,
\[
\limsup_{N\to\infty}\mathbb{P}^{\ur,\us;T}_{a_N,+,\lambda_N}
\bigl( \sup_{s\in[0,\delta]}
|\ux^{\lambda_N}(s) - \ux^{\lambda_N}(0)|
\geq \gamma \bigr)
\leq
e^{-c_1 M} + e^{c_2 M-\gamma^2/4
\delta} .
\]
Choosing $M=\frac{\gamma^2}{8c_2 
\delta}$ (and assuming that the parameters are tuned in such a way that 
$M\geq 2T$), we finally obtain
\[
\limsup_{N\to\infty} 
\mathbb{P}^{\ur,\us;T}_{a_N,+,\lambda_N}
\bigl( \sup_{s\in[0,\delta]}
|\ux^{\lambda_N}(s) - \ux^{\lambda_N}(0)|
\geq \gamma \bigr)
\leq
2 e^{-c_3\gamma^2/8
\delta} ,
\]
where $c_3=\min\{\frac{c_1}{c_2},1\}$. Thus, \eqref{tightness.1} is proved.

\section{Proof of Theorem~\ref{prop:mix}} 
\label{sec:mix} 

Throughout this section, we shall assume that $\Hla^2\in\bbN$; this implies 
that $\bbZ\subset\bbZ_\lambda$. In particular, the values of rescaled walks 
$\ux^\lambda(\ell)$ in discrete $\bbZ_\lambda$-time are well defined for any 
$\ell\in\bbZ$.

\subsection{Regular set $\Anr$, regular intervals and good blocks.}

Fix $\eta<\infty$ large enough and $\epsilon>0$ small enough.
The regular subset $\Anr\subset\An$ is defined as (under the convention that
$x_0\equiv 0$):
\be{eq:Anr} 
\Anr = \bigl\{ \ux\in\An \,:\, x_n \leq \eta
\quad\text{and}\quad
\min_{i\leq n} (x_i -x_{i-1}) \geq \epsilon \bigr\} .
\ee
The notion of regular interval is defined relative to a given continuous
$\An$-valued function $\ux(\cdot)$. Namely, an interval $[\ell,\ell +1]$ is
said to be \emph{regular} if
\be{eq:Dl-reg} 
\ux(\ell), \ux(\ell+1) \in \Anr
\quad\text{and}\quad
\max_{t\in [\ell,\ell+1]} x_n(t) \leq 2\eta .
\ee
Consider now the intervals $D_\ell = [2\ell,2(\ell+1)]$, which we shall call 
\emph{blocks}. A block is a union of two successive unit length intervals, 
\[
D_\ell = [2\ell,2\ell+1] \cup [2\ell+1,2(\ell+1)] \df D_\ell^-\cup D_\ell^+ .
\]
We shall say that $D_\ell$ is \emph{good} if both $D_\ell^+$ and $D_\ell^-$ 
are regular. If the notion of \emph{goodness} is defined with respect to random 
trajectories, namely $\ux^\lambda(\cdot)$, then we shall also use $D_\ell$ for 
the corresponding event.
\begin{lemma} 
\label{lem:good} 
Define $\Anlr = \Anr\cap\Anl$. There exist two constants $c_1,c_2$ such that
\be{eq:good}
c_1 h_\lambda^n
\leq
\bbP_{0,2,+,\lambda}^{\ur,\uz} \bigl(\ux^\lambda(1)=\us \bigm| D_0 \bigr) 
\leq
c_2 h_\lambda^n , 
\ee
uniformly in $\lambda$ small and in $\ur,\us,\uz\in\Anlr$. 
\end{lemma} 
We prove Lemma~\ref{lem:good} in Subsection~\ref{sub:good}. 

\subsection{Good blocks for a couple of trajectories.}

Consider now a couple of independent trajectories 
$\bigl(\ux^\lambda(\cdot),\uy^\lambda(\cdot)\bigr)$ (rescaled as 
in~\eqref{eq:xN}), distributed according to
\[ 
\bbP_{a,+,\lambda}^{\ur,\us}\otimes\bbP_{b,+,\lambda}^{\uu,\uw} .
\] 
Set $3M=\min\{a,b\}$. For $D_\ell \subset [-2M,2M]$, let us define
\be{eq:BlandM}
\frD_\ell =
\{ \text{$D_\ell$ is good for both $\ux^\lambda$ and $\uy^\lambda$} \}
\quad\text{and}\quad
\calM_0 = \;\;\sum_{\mathclap{-M \leq \ell \leq M-1}}\;\; \1_{\frD_\ell} .
\ee
\begin{lemma} 
\label{lem:calM} 
There exist $\nu>0$ and $\kappa>0$ such that
\be{eq:calM}
\bbP_{a,+,\lambda}^{\ur,\us}\otimes\bbP_{b,+,\lambda}^{\uu,\uw}
( \calM_0 \leq \nu M )
\leq \mathrm{e}^{-\kappa M} ,
\ee
uniformly in $\lambda$ small, $M$ large and $r_n,s_n,u_n,v_n\leq C$.
\end{lemma}
The proof of Lemma~\ref{lem:calM} is relegated to Subsection~\ref{sub:calM}. 

\subsection{A coupling argument.}

Fix $\lambda$ small, a negative integer $a<-2T$ and $\ur\in\Anlr$.
For $K\in\bbN$ and $\uu,\uv\in\Anlr$, define
\be{eq:Qmeasure} 
\bbQ_K^{\uu,\uv} (\cdot)
=
\bbP_{a,2K+1,+,\lambda}^{\ur,\uu}\otimes\bbP_{a,2K+1,+,\lambda}^{\ur,\uv}
\bigl(\,\cdot \bigm| \frD_K^- \bigr) ,
\ee
where, similarly to~\eqref{eq:BlandM}, we define
\[
\frD_\ell^\pm
=
\{ \text{$D_\ell^\pm$ is regular for both $\ux^\lambda$ and $\uy^\lambda$} \} . 
\]
In this way, $\frD_\ell = \frD_\ell^- \cap \frD_\ell^+$.
As before, the number $\calM_0$ of good blocks $D_\ell$ for $\ell\in 
\{1,\ldots,K-1\}$ is defined by
\be{eqMnot}
\calM_0 = \sum_{\ell=1}^{K-1} \1_{\frD_\ell} .
\ee
Let $\calF_T = \calF_T^\lambda$ be the $\sigma$-algebra generated by rescaled
trajectories~\eqref{eq:xN} on $[-2T,0]$. For the $\sigma$-algebra generated by 
a couple of such trajectories 
$\bigl(\ux^\lambda(\cdot),\uy^\lambda(\cdot)\bigr)$, we use 
$\calF_T\times\calF_T$.
Given $A\in\calF_T$, $A\times\Omega$ stands for the event that 
$\ux^\lambda(\cdot)\in A$ without further restrictions on
$\uy^\lambda(\cdot)$; $\Omega\times A$ is defined similarly. Define
\be{eq:psi-m}
\psi (m) = \sup_{K>m} \sup_{\uu,\uv\in\Anlr} \sup_{A\in\calF_T}
\bigl\{
\bbQ_K^{\uu,\uv}\bigl(A\times\Omega ; \calM_0 \geq m\bigr) -
\bbQ_K^{\uu,\uv}\bigl(\Omega\times A; \calM_0 \geq m\bigr) \bigr\} .
\ee
\begin{figure}[t]
\begin{center}
\resizebox{13.5cm}{!}{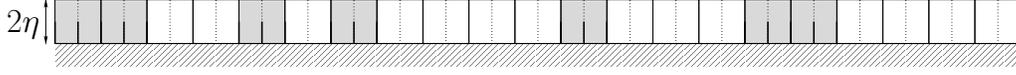}
\caption{\small Decomposition of the line into blocks. The shaded blocks 
represent jointly good blocks. Note that, in this case, the couple 
$(\ux^\lambda(\cdot),\uy^\lambda(\cdot))$ must be such that all trajectories 
stay inside the shaded area above jointly good blocks and cross the bold line 
segments in such a way that their $n$ paths $x^\lambda_1,\ldots,x^\lambda_n$, 
resp.~$y^\lambda_1,\ldots,y^\lambda_n$, are $\epsilon$-separated. Consequently, 
$\ux^\lambda(\cdot)$ and $\uy^\lambda(\cdot))$ can be coupled with positive 
probability, independently over each such jointly good block.}
\label{fig:goodBlocks}
\end{center}
\end{figure}%
\begin{lemma} 
\label{thm:delta} 
There exists $\delta>0$, which does not depend on $\lambda$, $a$ and $\ur$, 
such that
\be{eq:delta} 
\psi(m) \leq (1-\delta)^m .
\ee
\end{lemma}
\begin{proof}
The idea of the proof is sketched in Fig.~\ref{fig:goodBlocks}.

Let $K$, $\uu$ and $\uv$ be as above. Define
\[
\tau = \max\{ \ell<K \,:\, \frD_\ell \text{ occurs} \} .
\]
For any $\calA\in\calF_T\times\calF_T$, consider the decomposition
\be{eq:Decomp1} 
\bbQ_K^{\uu,\uv} \bigl(\calA; \calM_0 \geq m\bigr)
=
\sum_{\ell=1}^{K-1} \bbQ_K^{\uu,\uv} \bigl(\calA; \calM_0 \geq m ; \tau = 
\ell\bigr) .
\ee
In its turn, let us decompose each summand in~\eqref{eq:Decomp1} as
\be{eq:Decomp2} 
\bbQ_K^{\uu,\uv} \bigl(\calA ; \calM_0 \geq m ; \tau = \ell\bigr)
=
\sum_{\ux,\uy\in\Anlr}
\bbQ_K^{\uu,\uv} \bigl(\calA ; \calM_0 \geq m ; \tau = \ell ; \{\ux,\uy\}_\ell 
\big) , 
\ee
where we used the shorthand notation
\[
\{\ux,\uy\}_\ell
=
\bigl\{ \ux^\lambda(2\ell+1) = \ux \,;\, \uy^\lambda(2\ell+1) = \uy \bigr\} .
\]
Now
the 
Markov property implies that
\be{eq:Cond-xy}
\bbQ_K^{\uu,\uv} \bigl(\calA ; \calM_0 \geq m \bigm| \tau = \ell ; 
\{\ux,\uy\}_\ell \bigr)
= 
\bbQ_\ell^{\ux,\uy} \bigl(\calA; \calM_0 \geq m-1 \bigr) .
\ee 
{
Therefore, 
\begin{multline}
\label{eq:decomp-A}
\bbQ_K^{\uu,\uv} \bigl(\calA; \calM_0 \geq m\bigr)
=
\sum_{\ell=1}^{K-1} \sum_{\ux , \uy\in \Anlr} 
\bbQ_\ell^{\ux,\uy} \bigl(\calA; \calM_0 \geq m-1 \bigr)\\
\times 
\bbQ_K^{\uu,\uv} \bigl( \{\ux,\uy\}_\ell ~\big|~ \tau = \ell\bigr) 
\bbQ_K^{\uu,\uv} \bigl(\tau = \ell\bigr) . 
\end{multline}
This means that, for any $A\in\calF_T$, and for any $m,K,\uu,\uv$ in question, 
the following holds:
\begin{multline} 
\label{eq:sum-l-decomp}
\bbQ_K^{\uu,\uv} \bigl(A\times \Omega ; \calM_0 \geq m\bigr) - 
\bbQ_K^{\uu,\uv} \bigl( \Omega \times A ; \calM_0 \geq m\bigr)\\ 
\quad = 
\sum_{\ell , \ux , \uy} \bigl\{ 
\bbQ_\ell^{\ux,\uy} \bigl( A\times\Omega ; \calM_0 \geq m-1 \bigr) - 
\bbQ_\ell^{\ux,\uy} \bigl(\Omega\times A; \calM_0 \geq m-1 \bigr)
\bigr\} 
\\ 
\times 
\bbQ_K^{\uu,\uv} \bigl( \{\ux,\uy\}_\ell ~\big|~ \tau = \ell\bigr) 
\bbQ_K^{\uu,\uv} \bigl(\tau = \ell\bigr) .
\end{multline}
Since, evidently,
\be{eq:sym}
\bbQ_\ell^{\ux,\uy} \bigl( A\times\Omega  ; \calM_0 \geq m-1 \bigr)
= 
\bbQ_\ell^{\uy,\ux} \bigl( \Omega\times A ; \calM_0 \geq m-1 \bigr) ,
\ee
all the terms with $\ux = \uy$ in \eqref{eq:sum-l-decomp} vanish. On the other 
hand, each unordered pair $\ux\neq \uy$ is encountered exactly twice. Hence, 
again in view of~\eqref{eq:sym}, the contribution of each \emph{unordered 
pair} $\ux\neq \uy$ to the right-hand side of~\eqref{eq:sum-l-decomp} is equal 
to 
\begin{multline}
\label{eq:dif-Ql}
\bigl\{
\bbQ_\ell^{\ux,\uy} \bigl( A\times\Omega ; \calM_0 \geq m-1 \bigr) - 
\bbQ_\ell^{\ux,\uy} \bigl( \Omega\times A  ; \calM_0 \geq m-1 \bigr)
\bigr\}\\
\times 
\bigl\{
\bbQ_K^{\uu,\uv} \bigl( \{\ux,\uy\}_\ell ~\big|~ \tau = \ell\bigr) - 
\bbQ_K^{\uu,\uv} \bigl( \{\uy,\ux\}_\ell ~\big|~ \tau = \ell\bigr)
\bigr\}. 
\end{multline}
On the other hand, by Lemma~\ref{lem:good},
\begin{equation}\label{eq:q-coef} 
c_1 
\leq 
\frac{
\bbQ_K^{\uu,\uv} \bigl( \{\ux,\uy\}_\ell \bigm| \tau = \ell \bigr) 
}
{
\bbQ_K^{\uu,\uv} \bigl( \{\uy,\ux\}_\ell \bigm| \tau = \ell \bigr) 
}
\leq
c_2,
\end{equation}
uniformly in all the situations in question. Set $\delta = \frac{c_1}{c_2}$. 
Then, \eqref{eq:q-coef} implies that the expression in~\eqref{eq:dif-Ql}
is bounded above by 
\begin{multline}\label{eq:ux-bound} 
\psi(m-1) (1-\delta) \max \bigl\{
\bbQ_K^{\uu,\uv} \bigl( \{\ux,\uy\}_\ell \bigm| \tau = \ell \bigr) , 
\bbQ_K^{\uu,\uv} \bigl( \{\uy,\ux\}_\ell \bigm| \tau = \ell \bigr)
\bigr\}\\ 
\leq 
\psi(m-1) (1-\delta) \bigl\{
\bbQ_K^{\uu,\uv} \bigl( \{\ux,\uy\}_\ell \bigm| \tau = \ell\bigr) + 
\bbQ_K^{\uu,\uv} \bigl( \{\uy,\ux\}_\ell \bigm| \tau = \ell\bigr)
\bigr\} .
\end{multline}
Since $\bbQ_K^{\uu,\uv}$ is a probability measure, 
substituting~\eqref{eq:ux-bound} into \eqref{eq:sum-l-decomp} yields the 
conclusion \eqref{eq:delta} of the lemma. 
}
\end{proof}

\subsection{Conclusion of the proof.}

We are in a position to conclude the proof of Theorem~\ref{prop:mix}.
Let $a,b\geq (T+K)$ and $\ur,\us,\uw,\uz\in\Anl$ with $r_n,s_n,w_n,z_n\leq C$.
Let $A$ be an event generated by the rescaled trajectories of~\eqref{eq:xN} on 
$[-T,T]$. Then, 
\[
\bbP_{a,+,\lambda}^{\ur,\us}(A) - 
\bbP_{b,+,\lambda}^{\uw,\uz}(A) = 
\bbP_{a,+,\lambda}^{\ur,\us}\otimes\bbP_{b,+,\lambda}^{\uw,\uz}
( \1_{A\times\Omega} - \1_{\Omega\times A} ) .
\]
Consider a pair of trajectories 
$\bigl(x^\lambda(\cdot),y^\lambda(\cdot)\bigr)$, sampled from
$\bbP_{a,+,\lambda}^{\ur,\us}\otimes\bbP_{b,+,\lambda}^{\uw,\uz}$.
Let $\calM_{\pm}$ be the number of jointly good blocks $D_\ell$ with
$T\leq 2\ell \leq T+K-2$, respectively $-T-K \leq 2\ell \leq -T-2$.
 
By Lemma~\ref{lem:calM}, there exist $\nu^\prime,\kappa^\prime>0$ such that,
up to the $2 \mathrm{e}^{-\kappa^\prime K}$ correction, we may restrict our
attention to the event
\[
E_K = \{ \calM_+ \geq \nu^\prime K \} \cap \{ \calM_- \geq \nu^\prime K\} .
\]
On the other hand, by Lemma~\ref{thm:delta}, 
\[
\bigl| \bbP_{a,+,\lambda}^{\ur,\us}\otimes\bbP_{b,+,\lambda}^{\uw,\uz}
\bigl( \1_{E_K} (\1_{A\times\Omega} - \1_{\Omega\times A}) \bigr) \bigr|
\leq
(1-\delta)^{\nu^\prime K} .
\]
Our target exponential mixing bound~\eqref{eq:mix} follows.

\section{Proof of Lemmas~\ref{lem:good} and~\ref{lem:calM}}
\label{sec:lems}

\subsection{Probabilistic estimates}

Our proofs of Lemma~\ref{lem:good} and~\ref{lem:calM} rely on strong 
approximation techniques and on refined information on random walks in Weyl
chambers. There are three inputs, \eqref{eq:Inp1}-\eqref{eq:Inp3}, which are 
stated below, but proved in the Appendix. In the sequel, we fix $\eta$ 
sufficiently large; in particular, $\eta>C$, where $C$ is the constant which 
appears in Theorem~\ref{thm:A}. Furthermore, we fix $\epsilon>0$ sufficiently 
small. 

First of all, we claim that, for any $0<a<b<\infty$, there exists 
$\nu=\nu(a,b)>0$ such that
\be{eq:Inp1}
\tag{\textbf{I.1}}
\hat{\mathbf{P}}^{\ur}_{t,+,\lambda}(
\ux^\lambda(t) = \uz,
\max_{s\in [0,t]} x_n^\lambda(s) \leq 2\eta )
\geq
\nu h_\lambda^{n} ,
\ee
uniformly in $t\in [a,b]$, $\ur,\uz\in\Anlr$ and $\lambda$
small. 

Next, let
\be{eq:def-tau-time} 
\tau = \inf\{t\geq 0 
\,:\, \ux^\lambda\not\in\Anl\} 
\ee
be the first exit time of the path from $\Anl$.
We then claim that, for any $0<a<b<\infty$, there exists $\rho=\rho(a,b)$ such 
that the following two lower bounds hold uniformly in $t\in [a,b]$, 
$\uu\in\Anl$ with $u_n\leq\eta$ and in $\lambda$ sufficiently small:
\begin{gather}
\label{eq:Inp2}
\tag{\textbf{I.2}}
\hat{\mathbf{P}}^{\uu}_{t,\lambda}
\bigl( \max_{s\in[0,t]} x_n^\lambda(s) \leq 2\eta \bigm| \tau > t \bigr) 
\geq \rho ,\\ 
\label{eq:Inp3} 
\tag{\textbf{I.3}} 
\hat{\mathbf{P}}^{\uu}_{t,+,\lambda}
\bigl( \ux^\lambda(t) \in \bbA^{+,r}_{n,\lambda} \bigm|
\max_{s\in[0,t]} x_n^\lambda(s) \leq 2\eta\,,\,\tau>t \bigr) \geq \rho .
\end{gather}

\subsection{Proof of Lemma~\ref{lem:good}.} 
\label{sub:good}

Since, by definition, the area-tilt of every path in $D_0$ is uniformly 
bounded, it suffices to prove the lemma for random walks without area tilts. 
That is, we need to show that
\begin{equation}
\label{L3.4:1}
c_1 h_\lambda^n
\leq 
\hat{\mathbf{P}}^{\ur,\us}_{2,+,\lambda}
(\ux^\lambda(1) = \uz \,|\, D_0)
\leq
c_2 h_\lambda^n ,
\end{equation}
uniformly in $\ur,\us,\uz\in\Anlr$.  

We start by noting that upper bounds for
$\hat{\mathbf{P}}^{\ur}_{+,\lambda} (D_0 , \ux^\lambda(2) = \us)$ and
$\hat{\mathbf{P}}^{\ur}_{+,\lambda} (\ux^\lambda(1) = \uz , D_0 ,
\ux^\lambda(2) = \us)$ follow from the classical inequalities for concentration 
functions. Indeed, by~\cite[Theorem~6.2]{Esseen1968}, there exists a constant 
$c_3$ such that
\begin{equation}
\label{eq:conc-bound}
\hat{\mathbf{P}}^{\ur}_{+,\lambda} (\ux^\lambda(t) = \uy)
\leq
\hat{\mathbf{P}}^{\ur}_\lambda (\ux^\lambda(t) = \uy)
\leq
\frac{c_3}{t^{n/2}} h_\lambda^n ,
\end{equation}
uniformly in $\ur,\uy$ and $t\geq 0$.

Consequently,
\begin{equation}
\label{L3.4:2}
\hat{\mathbf{P}}^{\ur}_{+,\lambda} (D_0 , \ux^\lambda(2) = \us)
\leq
\hat{\mathbf{P}}^{\ur}_\lambda (\ux^\lambda(2) = \us)
\leq
\frac{c_3}{2^{n/2}} h_\lambda^n
\end{equation}
and
\begin{align}
\hat{\mathbf{P}}^{\ur}_{+,\lambda} (\ux^\lambda(1) = \uz , D_0 ,
\ux^\lambda(2) = \us)
&\leq
\hat{\mathbf{P}}^{\ur}_\lambda (\ux^\lambda(1) = \uz , \ux^\lambda(2) = \us) 
\notag\\
&=
\hat{\mathbf{P}}^{\ur}_\lambda (\ux^\lambda(1) = \uz)
\hat{\mathbf{P}}^{\uz}_\lambda (\ux^\lambda(1) = \us) \notag\\
&\leq
c_3^2 h_\lambda^{2n} ,
\label{L3.4:3}
\end{align}
uniformly in $\ur,\us,\uz\in\Anlr$.

The corresponding matching lower bounds follow from~\eqref{eq:Inp1}. 
Indeed, 
\begin{align}
&\hat{\mathbf{P}}^{\ur}_{+,\lambda} (\ux^\lambda(1) = \uz, D_0 ,
\ux^\lambda(2) = \us) \notag\\
&\hspace{2cm}=
\hat{\mathbf{P}}^{\ur}_{+,\lambda} (\ux^\lambda(1) = \uz ,
\max_{t\leq 1} x_n^\lambda(t) \leq 2\eta)
\hat{\mathbf{P}}^{\uz}_{+,\lambda} (\ux^\lambda(1) = \us ,
\max_{t\leq 1} x_n^\lambda(t) \leq 2\eta) \notag\\
&\hspace{2cm}\stackrel{\eqref{eq:Inp1}}{\geq}
c_4 h_\lambda^{2n} ,
\label{L3.4:7}
\end{align}
for any $\ur,\uz,\us\in\Anlr$. Since the cardinality
\be{eq:Areg-size} 
\abs{\Anlr} \geq c_5(\epsilon) \eta^n h_\lambda^{-n},
\ee
we infer, by summing over $\uz$ in~\eqref{L3.4:7}, that
\begin{equation}
\label{L3.4:6}
\hat{\mathbf{P}}^{\ur}_{+,\lambda} (\ux^\lambda(2) = \us , D_0)
\geq
c_6 h_\lambda^n .
\end{equation}
It remains to note that the lower bound in~\eqref{L3.4:1} follows 
from~\eqref{L3.4:2} and~\eqref{L3.4:7}, and that the upper bound 
in~\eqref{L3.4:1} follows from~\eqref{L3.4:3} and~\eqref{L3.4:6}.

\subsection{Proof of Lemma~\ref{lem:calM}.} 
\label{sub:calM}
The proof of Lemma~\ref{lem:calM} proceeds in two steps.

Consider the $5$-blocks 
\[
D^{(5)}_\ell \df
D_{5\ell-2} \cup D_{5\ell-1} \cup D_{5\ell} \cup D_{5\ell+1} \cup D_{5\ell+2} ,
\]
where $\ell\in\{-\lfloor M/5 \rfloor,\ldots,\lfloor M/5 \rfloor\}\subset\bbZ$. 

Let us say that a $5$-block  $D^{(5)}_\ell$ is pre-good (relative to a 
trajectory $\ux^\lambda(\cdot)$) if both
\be{eq:5pre-good} 
\min_{t\in D_{5\ell-2}} x^\lambda_n(t),\,
\min_{t\in D_{5\ell+2}} x^\lambda_n(t) \leq \eta .
\ee
Given a couple of trajectories $\ux^\lambda$ and $\uy^\lambda$, let 
$\tilde\frD_\ell^{(5)}$ denote the event that $D^{(5)}_\ell$ is pre-good for 
both $\ux^\lambda$ and $\uy^\lambda$. 

\smallskip
\noindent
\step{1} 
Note that the definitions are set up in such a way that $D_{5\ell}$ is the 
middle section of $D^{(5)}_\ell$. We claim that there exists $\rho_1 = 
\rho_1(\eta,\epsilon)>0$ such that
\be{eq:pre-to-good1} 
\bbP_{-4,6,+,\lambda}^{\ur,\us}
\bigl( \text{$D_0$ is good} \bigm| \text{$D_0^{(5)}$ is pre-good} \bigr)
\geq \rho_1 , 
\ee
uniformly in $\ur,\us\in\Anl$ and $\lambda$ sufficiently small. 

As a result, for any
$\ell\in \{-\lfloor M/5 \rfloor,\ldots,\lfloor M/5 \rfloor\}\subset\bbZ$,  
\be{eq:Pre-to-good}
\bbP_{a,+,\lambda}^{\ur,\us}\otimes\bbP_{b,+,\lambda}^{\uu,\uw}
\bigl( \frD_{5\ell} \bigm| \tilde\frD_{\ell}^{(5)} \bigr) \geq \rho_1^2 , 
\ee
uniformly in $\ur,\us,\uu,\uv\in\Anl$ and $\lambda$ small enough.
 
By the Markov property, this means that any jointly pre-good $5$-block gives 
rise to a good block in its middle section with probability at least 
$\rho_1^2$, regardless of the behavior of trajectories outside this particular 
pre-good $5$-block.

\smallskip
\noindent 
\step{2} In this second step, we control the density of jointly pre-good 
$5$-blocks $D_{\ell}^{(5)}$ which lie inside $[-2M,2M]$. Define
\[
\calM_0^{(5)} = 
\sum_{j=-\lfloor M/5 \rfloor}^{\lfloor M/5 \rfloor} \1_{\frD_\ell^{(5)}} .
\]
We claim that there exist $\nu^{(5)}>0$ and $\kappa^{(5)}>0$ such that
\be{eq:calMtilde} 
\bbP_{a,+,\lambda}^{\ur,\us}\otimes\bbP_{b,+,\lambda}^{\uu,\uw}
\bigl( \calM_0^{(5)} \leq \nu^{(5)} M \bigr)
\leq 
\mathrm{e}^{-\kappa^{(5)} M} .
\ee
uniformly in $\lambda$ small, $M$ large, $a,b\geq 3M$  and $r_n,s_n,u_n,v_n\leq 
C$.

Evidently, \eqref{eq:Pre-to-good} and~\eqref{eq:calMtilde} imply the target 
bound~\eqref{eq:calM}.

\subsection{Proof of \eqref{eq:pre-to-good1}.} 

\begin{figure}[t]
\begin{center}
\resizebox{13.5cm}{!}{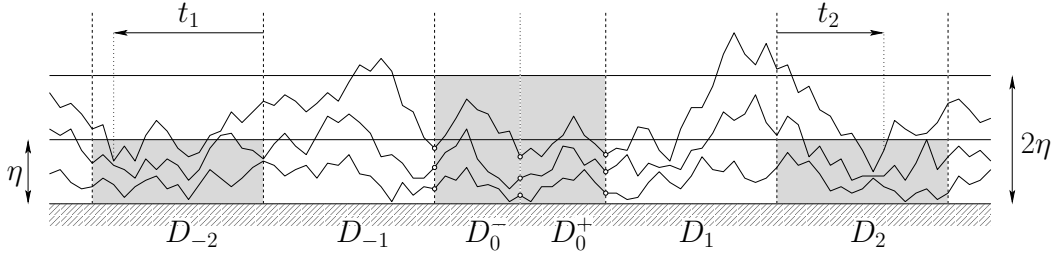}
\caption{\small A picture of the $5$-block $D_0^{(5)}$. In the picture, this 
$5$-block is pre-good, since the top-most path (here $n=3$) visits the shaded 
areas in the blocks $D_{-2}$ and $D_2$; the corresponding random variables 
$t_1$ and $t_2$ are also represented. The event $D_0$ also occurs: the top-most 
path stays inside the shaded area above the block $D_0=D_0^-\cup D_0^+$ and 
the $n=3$ paths stay $\epsilon$ apart from each other and the bottom wall at 
the boundary of $D_0^-$ and $D_0^+$ (the corresponding positions of the $3$ 
paths there are marked with dots).}
\end{center}
\label{fig:5-block}
\end{figure}

We are going to show that
\be{eq:new-bound-block} 
\bbP_{-2-T_1,4+T_2,+,\lambda}^{\uu,\uv}( D_0 )
\geq \rho_1 , 
\ee
uniformly in $\uu,\uv$ with $u_n,v_n\leq\eta$, $T_1,T_2\in [0,2]$ and $\lambda$ 
sufficiently small.
The target~\eqref{eq:pre-to-good1} is an immediate consequence by the Gibbs 
property and conditioning on the left-most $t_-\in [-4,-2]$ and the right-most 
$t_+\in [4, 6]$ such that $x^\lambda_n(t_-),x^\lambda_n(t_+)\leq\eta$; see 
Fig.~\ref{fig:5-block}.

The proof boils down to deriving an appropriate upper bound on the 
partition function $Z_{-2-T_1,4+T_2,+,\lambda}^{\uu,\uv}$ and an appropriate 
matching lower bound on the constrained partition function
$Z_{-2-T_1,4+T_2,+,\lambda}^{\uu,\uv}[D_0]$.

In the sequel, $\tilde{\ux}^\lambda$ stands for the 
{
{\em reversed}  
}
random walk with 
transition probabilities $\tilde{p}_z = p_{-z}$.
Let $\tilde{\tau}$ be the first exit time of $\tilde{\ux}^\lambda$ from 
$\Anlr$.
We assume that the constants $\nu$ and $\rho$ in the probabilistic 
estimates~\eqref{eq:Inp1}-\eqref{eq:Inp3} are chosen in such a way that the 
corresponding bounds hold for the reflected process as well.

\smallskip
\noindent 
\step{1} (An upper bound on $Z_{-2-T_1,4+T_2,+,\lambda}^{\uu,\uv}$.)
Since we are dealing with non-negative potentials,
\begin{align}
Z_{-2-T_1,4+T_2,+,\lambda}^{\uu,\uv}
&\leq 
\hat{\mathbf P}_{2+T_1,\lambda}^{\uu} ( \tau > 2+T_1 )
\max_{\ur,\us\in\Anl} \hat{\mathbf P}^{\ur}_{2,\lambda} ( \ux^\lambda(2) = \us )
\, \hat{\mathbf P}_{2+T_2,\lambda}^{\uv} ( \tilde \tau > 2+T_2 ) \notag\\
&\leq
c_1 h_\lambda^n
\hat{\mathbf P}_{2+T_1,\lambda}^{\uu} ( \tau > 2+T_1 )
\, \hat{\mathbf P}_{2+T_2,\lambda}^{\uv} ( \tilde \tau > 2+T_2 ) .
\label{eq:ub-pf-new}
\end{align}
The second inequality follows from the concentration 
bound~\eqref{eq:conc-bound}. 

\smallskip
\noindent
\step{2} (A lower bound on $Z_{-2-T_1,4+T_2,+,\lambda}^{\uu,\uv}[D_0]$.)
By our assumption~\eqref{eq:HL-3},
\begin{multline}
\label{eq:lb-new1} 
Z_{-2-T_1,4+T_2,+,\lambda}^{\uu,\uv}(D_0) \geq \\
\mathrm{e}^{-10 n q_0 (2\eta)}
\mathbf{P}_{-2-T_1,4+T_2,+,\lambda}^{\uu}
\bigl( D_0 , \,\;\quad\max_{\mathclap{t\in [-2-T_1,4+T_2]}}\quad\;\, 
x^\lambda_n(t) \leq 2\eta , \ux^\lambda(4+T_2) = \uv \bigr) .
\end{multline}
Above $\hat{\mathbf P}^{\uv}_{s,t,+,\lambda}$ is the provisional notation
for the restriction of the law of the rescaled walk started at time $s$
at $\uv$ to the set of trajectories which stay inside $\Anl$ during the time
interval $[s,t]$. 

The probability on the right-hand side of~\eqref{eq:lb-new1} is bounded below
by the following product of three factors:
\begin{align}
&\hat{\mathbf P}_{2+T_1,\lambda}^{\uu} \bigl(
\tau > 2+T_1 , \max_{t\leq 2+T_1} x^\lambda_n(t) \leq 2\eta ,  
\ux^\lambda(2+T_1)\in\Anlr \bigr) \notag\\ 
&\hspace*{1cm}\times
\min_{\ur,\us\in\Anlr} \hat{\mathbf P}_{2,+,\lambda}^{\ur} ( D_0 , 
\ux^{\lambda}(2) = \us ) \notag\\
&\hspace*{1cm}\times
\hat{\mathbf P}_{2+T_2,\lambda}^{\uv} \bigl(
\tilde\tau > 2+T_2 , \max_{t\leq 2+T_2} \tilde x^\lambda_n(t)\leq 2\eta , 
\tilde\ux^\lambda(2+T_2)\in\Anlr \bigr) .
\label{eq:lb-new2}
\end{align}
On the one hand, in view of~\eqref{L3.4:6}, the middle factor is bounded below 
by $c_2 h_\lambda^n$. 
On the other hand, the probabilistic bounds~\eqref{eq:Inp2}, \eqref{eq:Inp3} 
imply that the left-most factor in~\eqref{eq:lb-new2} is bounded below by
$\rho^2\hat{\mathbf P}_{2+T_1,\lambda}^{\uu} ( \tau > 2+T_1 )$.
Similarly, the right-most factor in~\eqref{eq:lb-new2} is bounded below by
$\rho^2 \hat{\mathbf P}_{2+T_2,\lambda}^{\uv} ( \tilde \tau > 2+T_2 )$.
Hence, 
\be{eq:lb-new-final}
Z_{-2-T_1, 4+T_2,+,\lambda}^{\uu,\uv}[D_0] \geq c_3 \rho^4 h_\lambda^n
\,\hat{\mathbf P}_{2+T_1,\lambda}^{\uu} ( \tau > 2+T_1 )
\,\hat{\mathbf P}_{2+T_2,\lambda}^{\uv} ( \tilde\tau > 2+T_2 ) .
\ee
Since
\[
\bbP_{-2-T_1,4+T_2,+,\lambda}^{\uu,\uv} ( D_0 )
= 
\frac{Z_{-2-T_1,4+T_2,+,\lambda}^{\uu,\uv}[D_0]} 
{Z_{-2-T_1,4+T_2,+,\lambda}^{\uu ,\uv}} , 
\]
\eqref{eq:new-bound-block} directly follows from~\eqref{eq:ub-pf-new} 
and~\eqref{eq:lb-new-final}.
\qed

\subsection{Proof of \eqref{eq:calMtilde}.} 

We start by deriving a lower bound on partition functions, as this will allow 
us to exclude sets of pathological trajectories.
\begin{lemma}
\label{lem:LB_PF}
There exist constants $c_1 = c_1(n)$ and $c_2 = c_2(n,\eta)$ and a sufficiently 
large value $T_0 = T_0(\eta)$ such that, for all $T\geq T_0$,
\begin{equation}
\label{eq:roughLB_Zmulti}
Z^{\uw,\uz}_{T,+,\lambda}
\geq
c_2\, e^{-c_1 T} \,
\hat{\mathbf P}_{2T,+,\lambda}^{\uw} ( \ux^\lambda (2T) = \uz ) ,
\end{equation}
uniformly in $\lambda$ small and in $z_n,w_n\leq\eta$.
\end{lemma}
\begin{proof} 
The point is that constant $c_1$ does not depend on $\eta$, only on the 
dimension $n$. The dependence of $c_1$ on $n$ is expressed in terms of 
the dependence of $\epsilon$ (in the definition of the regular set $\Anlr$, 
see~\eqref{eq:Anr}) on $n$. 
We shall work with a fixed small value of $\epsilon>0$ which satisfies
\be{eq:eps-bound}
n\epsilon < 1 . 
\ee
In the sequel, we consider $T>2$.  
Let $\Anlr(\alpha) = \Anlr \cap \{ \ux \,:\, x_n \leq \alpha \}$. Consider the
events
\begin{align}
\calE_- &= \bigl\{ \max_{t\in [0,1]} x_n^\lambda(t)\leq 2\eta \,,
\ux(1)\in\Anlr(1) \bigr\} , \notag\\
\label{eq:E-event}
\calE_{+} &= \bigl\{ \max_{t\in [2T-1,2T]} x_n^\lambda(t)\leq 2\eta \,, 
\ux(2T-1)\in\Anlr(1) \bigr\}
\intertext{and}
\calE_T &= \bigl\{ \max_{t\in [1,2T-1]} x_n^\lambda(t)< 2 \bigr\} . \notag
\end{align}
On the one hand, by~\eqref{eq:HL-3},  
\begin{equation}
\label{eq:lb-pf-1} 
Z^{\uw,\uz}_{T,+,\lambda}
\geq
\mathrm{e}^{- 2n q_0(2\eta) - 2n q_0(2) T} \,
\hat{\mathbf P}_{2T,+,\lambda}^{\uw} \bigl( \calE_{-} , \calE_T , \calE_+ ,
\ux^\lambda(2T) = \uz \bigr) . 
\end{equation}
On the other hand, 
\be{eq:ub-prob-2T} 
\hat{\mathbf P}_{2T, + , \lambda}^{\uw}\lb \ux^\lambda (2T) = \uz\rb 
\leq c_3 \hat{\mathbf P}_{1, \lambda}^{\uw}\lb \tau >1 \rb 
 \hat{\mathbf P}_{1, \lambda}^{\uz}\lb \tilde \tau >1 \rb  
\frac{h_\lambda^n}{T^{n/2}} .
\ee
Above, we relied on the concentration bound~\eqref{eq:conc-bound}. 

In order to compare the probabilities appearing in~\eqref{eq:lb-pf-1} 
and~\eqref{eq:ub-prob-2T}, note that an application of 
\eqref{eq:Inp1}-\eqref{eq:Inp3} (and the observation that, as 
in~\eqref{eq:Areg-size}, the cardinality $\abs{\Anlr(1)} \geq c_4(\epsilon) 
h_\lambda^{-n}$) yields
\begin{multline}
\label{eq:lb-tube1}
\hat{\mathbf P}_{2T,+,\lambda}^{\uw}
(\calE_{-} , \calE_T , \calE_+ , \ux^\lambda(2T) = \uz )
\geq
c_5 \hat{\mathbf P}_{1,\lambda}^{\uw} ( \tau > 1 )
\hat{\mathbf P}_{1,\lambda}^{\uz} ( \tilde\tau > 1 ) \\ 
\times 
\min_{\uu,\uv\in\Anlr(1)} \hat{\mathbf P}_{2(T-1),+,\lambda}^{\uu}
\bigl( \max_{t\in [0,2(T-1)]} x^\lambda_n(t) < 2 \,,\, 
\ux^\lambda(2(T-1)) = \uv \bigr) .
\end{multline}
However, 
\begin{equation}
\label{eq:lb-tube2} 
\min_{\uu,\uv\in\Anlr(1)} \hat{\mathbf P}_{2(T-1),+,\lambda}^{\uu}
\bigl(
\max_{t\in [0,2(T-1)]} x^\lambda_n(t) < 2 \,,\, \ux^\lambda(2(T-1)) = \uv
\bigr)
\geq
\mathrm{e}^{-c_6(\epsilon) T} h_\lambda^n .
\end{equation}
Indeed, consider $n$ walks $x^\lambda_\ell$, $\ell=1,\ldots,n$, which go 
from $u_\ell$ to $v_\ell$ inside space-time tubes of width $\epsilon/4$ 
centered around the space-time segments $[(u_\ell,0),(v_\ell,2(T-1)]$.
By construction, these walks stay in $\Anl \cap \{\ux \,:\, x_n < 2\}$. By a
coarse splitting into time-blocks of lengths of order $\epsilon^2$, we bound 
from below the probability of staying within such tubes by 
$\mathrm{e}^{-c_6(\epsilon) T}$. Applying the local CLT for the last step, we 
bound from below the probability of ending up in $\uz$ by a multiple of 
$h_\lambda^n$. \eqref{eq:lb-tube2} follows. 

\noindent
The bound~\eqref{eq:roughLB_Zmulti} is a direct consequence 
of~\eqref{eq:ub-prob-2T} and~\eqref{eq:lb-pf-1}, \eqref{eq:lb-tube1} 
and~\eqref{eq:lb-tube2}.
\end{proof}
Let us resume the proof of~\eqref{eq:calMtilde}. Without loss of 
generality, we shall assume that $a=3M$ and $b\geq a$. In the sequel, the 
trajectory $\ux^\lambda$ is sampled from $\bbP_{a,+,\lambda}^{\ur,\us}$ and 
$\uy^\lambda$ is sampled from $\bbP_{b,+,\lambda}^{\uu,\uv}$. Recall that 
$r_n,s_n,u_n,v_n\leq C\leq\eta$.

In principle, $b$ can be much larger than $M$. Let us verify that one 
can restrict our attention to the case where $b$ is of the same order as $M$. 
Define the random variables $B_{\pm}\geq 0$ via (see 
Figure~\ref{fig:Bplusminus})
\begin{equation}
\label{eq:b-pm}
-2M - B_- = \max \{ t\leq -2M \,:\, y_n^\lambda(t) \leq \eta \}  
\end{equation}
and, accordingly,
$2M + B_+ = \min \{ t\geq 2M \,:\, y_n^\lambda(t) \leq \eta \}$.
\begin{figure}[t]
\begin{center}
\resizebox{13.5cm}{!}{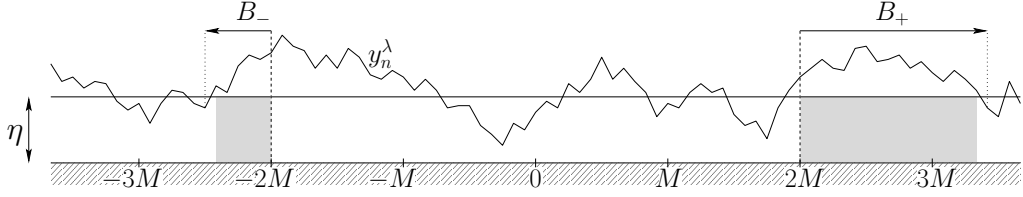}
\caption{\small The definition of the random variables $B_-$ and $B_+$. Note 
that the shaded areas are necessarily below the path $y^\lambda_n$ and thus 
contribute to the area-tilt.}
\end{center}
\label{fig:Bplusminus}
\end{figure}%
By the Gibbs property,
\begin{multline}
\bbP_{b,+,\lambda}^{\uu,\uv} ( B_\pm = b_\pm ) \\
\leq
\max_{w_n,z_n\leq\eta}
\bbP_{-2M-b_-,2M+b_+,+,\lambda}^{\uw,\uz} 
\bigl(
\min_{t\in (-2M-b_-,-2M]} x^\lambda_n(t) \wedge
\!\!\!\min_{t\in [2M,2M+b_+)}\!\! x^\lambda_n(t) >\eta \bigr) .
\label{eq:gibbs-bpm}
\end{multline}
Therefore, in view of~\eqref{eq:HL-3},
\begin{equation}
\label{eq:b-pm-fin1} 
\bbP_{b,+,\lambda}^{\uu,\uv} ( B_\pm = b_\pm )
\leq
\mathrm{e}^{-(b_- + b_+) q_0(\eta)}
\max_{w_n,z_n\leq\eta}
\frac{\hat{\mathbf P}_{T,+,\lambda}^{\uw} ( \ux^\lambda(T) = \uz )}
{\hat Z_{T,+,\lambda}^{\uw,\uz}} ,
\end{equation}
where we have set $T = 4M + b_- + b_+$. Using the lower 
bound~\eqref{eq:roughLB_Zmulti} on $\hat Z_{T,+,\lambda}^{\uw,\uz}$, we conclude 
that
\begin{equation}
\label{eq:b-pm-fin2}
\bbP_{b,+,\lambda}^{\uu,\uv} ( B_\pm = b_\pm )
\leq 
c_7(\epsilon) \mathrm{e}^{c_8(\epsilon) M - (b_-+b_+) q_0(\eta)} .
\end{equation}
Therefore, if we choose $\eta$ so large that
\begin{equation}
\label{eq:eta-eps-M}
q_0(\eta) > 2 c_8(\epsilon) ,
\end{equation}
then we may ignore the case $b_\pm \geq M$.

\smallskip
\noindent 
Consequently, \eqref{eq:calMtilde} will follow once we check that
\begin{equation}
\label{eq:calMtilde-M}
\bbP_{3M,+,\lambda}^{\ur,\us}\otimes\bbP_{-b_1,b_2,+,\lambda}^{\uu,\uw}
( \calM_0^{(5)} \leq \nu^{(5)} M )
\leq
\mathrm{e}^{- \kappa^{(5)} M} ,
\end{equation}
uniformly in $\lambda$ small, $M$ large, $b_1,b_2\in [2M,3M]$ and 
$r_n,s_n,u_n,v_n\leq\eta$.

If we choose $\nu^{(5)}$ to be sufficiently small, for instance smaller than 
$0.2$, then, by~\eqref{eq:HL-3},
\begin{multline}
\label{eq:calM-bd1}
\bbP_{3M,+,\lambda}^{\ur,\us}\otimes\bbP_{b,+,\lambda}^{\uu,\uw}
( \calM_0^{(5)} \leq \nu^{(5)} M ) \\
\leq
\mathrm{e}^{- \frac{q_0(\eta)}{10} M}
\frac{
\hat{\mathbf P}^{\ur}_{6M,+,\lambda} ( \ux^\lambda(3M) = \us )
\hat{\mathbf P}^{\uu}_{b_1+b_2,+,\lambda} ( \uy^\lambda(b) = \uw )
}{
Z^{\ur,\us}_{3M,+,\lambda} Z^{\uu,\uw}_{-b_1,b_2,+,\lambda}
} .
\end{multline}
Taking $\eta$ and $M$ large enough and applying~\eqref{eq:roughLB_Zmulti}, we
arrive to~\eqref{eq:calMtilde-M}.
\qed

\appendix

\section{Strong approximation techniques.}

In order to 
prove~\eqref{eq:Inp1}, we are going to apply strong approximation techniques
from~\cite{DW15}. By rescaling, it is sufficient to consider the case $t=1$.

In the sequel, $\hat{\mathbf{P}}^{\ur}_+$ denotes the restriction of the law of 
the $n$-dimensional Brownian motion $\uB$ started at $\ur$ to the set 
$\mathbb{A}^+_{n}$. 

Define
\[
O_\epsilon(\uz) = \{\uy \,:\, |y_i-z_i| \leq 
{
\frac{\epsilon}{3}
}
\text{ for all } i\} 
.
\]
It follows easily from~\cite[Lemma~17]{DW15} that
\begin{multline*}
\hat{\mathbf{P}}^{\ur}_{+,\lambda}
(\ux^\lambda(1-\gamma)\in O_\epsilon(\uz) ,
\max_{t\leq1-\gamma} x^\lambda_n(t) \leq 2\eta) \\
=
\hat{\mathbf{P}}^{\ur}_+ (\uB(1-\gamma) \in O_\epsilon(\uz) ,
\max_{t\leq 1-\gamma} B_n(t) \leq 2\eta) + o(1) ,
\end{multline*}
uniformly in $\ur,\uz\in\Anlr$. This implies that there exists a constant
$c(\epsilon,\eta,\gamma)>0$ such that
\begin{equation}
\label{L3.4:4}
\hat{\mathbf{P}}^{\ur}_{+,\lambda} (\ux^\lambda(1-\gamma) \in O_\epsilon(\uz) ,
\max_{t\leq 1-\gamma} x^\lambda_n(t) \leq 2\eta)
\geq
c(\epsilon,\eta,\gamma) ,
\end{equation}
for all $\ur,\uz\in\Anlr$. Since $O_\epsilon(\uz)$ is separated from the 
boundary of $\mathbb{A}_n^+$, 
we may choose $\gamma$ so small that the 
probability that the random walk $x^\lambda$ started at 
$\uy\in O_\epsilon(\uz)$ has, at time $\gamma$, the value $\uz$ and leaves 
$\mathbb{A}_{n,\lambda}^+$ before time $\gamma$ is quite small. This heuristic 
is made precise in~\cite[Lemma~29]{DW15}. In our notation, we can state that 
result as follows: 
There exist $a>0$ and  $c_1<\infty $ such that
\begin{multline*}
\hat{\mathbf{P}}^{\uy}_{+,\lambda} ( \ux^\lambda(\gamma) = \uz ,
\max_{t\leq\gamma} x^\lambda_n(t) \leq 2\eta ) \\
\geq
\hat{\mathbf{P}}^{\uy}_\lambda ( \ux^\lambda(\gamma) = \uz ,
\max_{t\leq\gamma} x^\lambda_n(t) \leq 2\eta )
- c_1 \gamma^{-n/2} e^{-a\epsilon^2/\gamma}h_\lambda^{n} .
\end{multline*}
By a similar argument, one can show that
\[
\hat{\mathbf{P}}^{\uy}_\lambda ( \ux^\lambda(\gamma) = \uz ,
\max_{t\leq\gamma} x^\lambda_n(t) \leq 2\eta ) \\
\geq
\hat{\mathbf{P}}^{\uy}_\lambda ( \ux^\lambda(\gamma) = \uz )
- c_2 \gamma^{-n/2} e^{-a\eta^2/\gamma} h_\lambda^{n} .
\]
Finally, by the standard local limit theorem,
\[
\hat{\mathbf{P}}^{\uy}_\lambda ( \ux^\lambda(\gamma) = \uz )
\geq c_3 \gamma^{-n/2} h_\lambda^{n} .
\]
As a result, we have the bound
\[
\hat{\mathbf{P}}^{\uy}_{+,\lambda} ( \ux^\lambda(\gamma) = \uz ,
\max_{t\leq\gamma} x^\lambda_n(t) \leq 2\eta)
\geq c_4 h_\lambda^{n} ,
\]
uniformly in $\uy,\uz\in\Anlr$.
Combining this bound with~\eqref{L3.4:4}, we infer that
\begin{equation}
\label{L3.4:5}
\hat{\mathbf{P}}^{\ur}_{1,+,\lambda} ( \ux^\lambda(1) = \uz ,
\max_{t\leq 1} x_n^\lambda(t) \leq 2\eta )
\geq c_5 h_\lambda^{n} ,
\end{equation}
uniformly in $\ur,\uz\in\Anlr$. \qed

\section{Invariance principles for random walks in Weyl chambers}

Conditional limit theorems and conditional invariance principles for random 
walks in different cones have been studied in~\cite{DW15} and~\cite{DW15b}. All 
the results in these papers are proved in the case when the non-rescaled walk 
starts at a fixed point. In this paragraph, we give certain improvements of
these results to the case when the starting point of the non-rescaled walk may
grow (but we shall consider walks in Weyl chambers only).

{
More precisely, we shall the following subsets of the euclidian space:
\begin{itemize}
 \item chamber of type $A$: $\{x: x_1<x_2<\ldots<x_n\}$;
 \item chamber of type $C$: $\{x: 0<x_1<x_2<\ldots<x_n\}$;
 \item chamber of type $D$: $\{x: |x_1|<x_2<\ldots<x_n\}$.
\end{itemize}
}

Let $u_W$ denote the unique (up to a constant multiplier) positive harmonic
function on $W$:
\begin{itemize}
\item if $W$ is the chamber of type $A$, then
$u_W(x) = \prod_{i<j} (x_j-x_i)$;
\item if $W$ is the chamber of type $C$, then
$u_W(x) = \prod_k x_k \prod_{i<j} (x^2_j-x^2_i)$;
\item if $W$ is the chamber of type $D$, then 
$u_W(x)=\prod_{i<j} (x^2_j-x^2_i)$.
\end{itemize}
\begin{proposition}
\label{inv.prop.1}
Let $W$ be a Weyl chamber of type $A$, $C$ or $D$. Let $\tau$ be the first exit 
time from $W$, that is,
\[
\tau = \inf \{t>0 \,:\, \ux^\lambda(t)\notin W\}.
\]
Then, as $\ur=\ur_\lambda\to 0$,
\[
\hat{\mathbf{P}}^{\ur}_\lambda ( \ux^\lambda(1)\in\,\cdot \;|\, \tau>1 )\to \mu
\quad\text{weakly},
\]
where $\mu$ is the probability measure on $W$ with density proportional to 
$u_W(x)e^{-|x|^2/2}$.

Furthermore, under $\hat{\mathbf{P}}^{\ur}_\lambda$, $\ux^\lambda$ converges
weakly on $C[0,1]$ to the Brownian meander in $W$ started at zero.
\end{proposition}
By ``Brownian meander in $W$'', we mean a Brownian motion conditioned on 
staying in $W$ up to time one. If the starting point lies inside $W$, then one 
has a condition of positive probability. However, if the starting 
point lies on the boundary of $W$, then the probability of the condition is zero
and it is not at all clear how one can construct such a process.
Garbit~\cite{Garbit2009} has constructed Brownian meanders started at zero for a 
quite large class of cones. This class includes Weyl chambers.
\begin{proof}
The main difference with~\cite[Theorem~3]{DW15} is that we find the limit
for conditional distributions without determining the asymptotic behavior of
$\hat{\mathbf{P}}^{\ur}_\lambda (\tau>1)$. (Recall once again that~\cite[Theorem 
3]{DW15} is proven under the assumption $\ur=h_\lambda a$ for some fixed $a\in 
W$.)

Fix some $\epsilon\in (0,1/2)$ and define the stopping time
\[
\nu_{\lambda,\epsilon} = \inf\{t>0 \,:\, \ux^\lambda(t)\in 
W_{\lambda,\epsilon}\} ,
\]
where
\[
W_{\lambda,\epsilon} =
\{x\in W \,:\, \operatorname{dist}(x,\partial W)\geq H_\lambda^{-2\epsilon}\} .
\]
According to~\cite[Lemma 14]{DW15},
\begin{equation}
\label{inv.1}
\hat{\mathbf{P}}^{\ur}_\lambda (\tau > H_\lambda^{-2\epsilon} ,
\nu_{\lambda,\epsilon} > H_\lambda^{-2\epsilon})
\leq e^{-c_1 H_\lambda^{2\epsilon}} ,
\end{equation}
uniformly in $\ur$. Since we consider lattice random walks, there exists $\ur_0$ 
such that
\[
\hat{\mathbf{P}}^{\ur}_\lambda (\tau>1) \geq 
\hat{\mathbf{P}}^{\ur_0}_\lambda (\tau>1) .
\]
(If $W$ is of type $A$ or $C$, then we may take $r_0=h_\lambda(1,2,\ldots,n)$, 
while if $W$ is of type $D$, then we may take $r_0=h_\lambda(0,1,\ldots,n-1)$.) 
According to~\cite[Theorem 1]{DW15},
\[
\hat{\mathbf{P}}^{\ur_0}_\lambda (\tau>1) \sim C_1 h_\lambda^p ,
\]
where $p$ is a positive constant depending on the type of $W$ only. 
Consequently,
\begin{equation}
\label{inv.2}
\hat{\mathbf{P}}^{\ur}_\lambda (\tau>1) \geq C_2 h_\lambda^p ,
\end{equation}
uniformly in $\ur$. Combining~\eqref{inv.1} and~\eqref{inv.2}, we infer that
\begin{equation}
\label{inv.3}
\frac{
\hat{\mathbf{P}}^{\ur}_\lambda (\tau > H_\lambda^{-2\epsilon} ,
\nu_{\lambda,\epsilon} > H_\lambda^{-2\epsilon})
}{
\hat{\mathbf{P}}^{\ur}_\lambda (\tau>1)
}
\to 0, \quad \lambda\downarrow 0 ,
\end{equation}
uniformly in $\ur$. 
Furthermore, it follows from the exponential Doob inequality that
\[
\hat{\mathbf{P}}^{\ur}_\lambda \bigl(
\max_{t\leq H_\lambda^{-2\epsilon}} 
|\ux^\lambda(t)-\ux^\lambda(0)| > \theta_\lambda \bigr)
\leq
e^{-c_2 \theta_\lambda^2 H_\lambda^\epsilon} ,
\]
where $\theta_\lambda\to 0$ sufficiently slowly. This implies that, whenever
$|\ur|\leq \theta_\lambda$,
\begin{equation}
\label{inv.4}
\frac{
\hat{\mathbf{P}}^{\ur}_\lambda \bigl(
\max_{t\leq H_\lambda^{-2\epsilon}} |\ux^\lambda(t)| > 2\theta_\lambda
\bigr)
}{
\hat{\mathbf{P}}^{\ur}_\lambda (\tau>1)
}
\to 0 .
\end{equation}
It follows now from~\eqref{inv.3} and~\eqref{inv.4} that, uniformly in $\ur$,
\begin{equation}
\label{inv.5}
\hat{\mathbf{P}}^{\ur}_\lambda (\tau>1) = (1+o(1)) \, 
\hat{\mathbf{P}}^{\ur}_\lambda (\tau>1 ,
\nu_{\lambda,\epsilon}\leq H_\lambda^{-2\epsilon} ,
\max_{t\leq\nu_{\lambda,\epsilon}} |\ux(t)| \leq 2\theta_\lambda)
\end{equation}
and
\begin{multline}
\label{inv.6}
\hat{\mathbf{P}}^{\ur}_\lambda (\ux^\lambda(1)\in A , \tau>1) \\
=
(1+o(1))\, \hat{\mathbf{P}}^{\ur}_\lambda
(\ux^\lambda(1)\in A , \tau>1 ,
\nu_{\lambda,\epsilon}\leq H_\lambda^{-2\epsilon} ,
\max_{t\leq\nu_{\lambda,\epsilon}} |\ux(t)|\leq 2\theta_\lambda)
\end{multline}
for any compact $A\subset W$.

Using the Markov property at time $\nu_{\lambda,\epsilon}$ and 
applying~\cite[Lemma~20]{DW15}, we obtain from~\eqref{inv.5} and~\eqref{inv.6}
\begin{align*}
\hat{\mathbf{P}}^{\ur}_\lambda (\tau>1) =
(c_3+o(1))\, h_\lambda^p\, \hat{\mathbf{E}}^{\ur}_\lambda 
\bigl[u_W(\ux^\lambda(\nu_{\lambda,\epsilon})) \,;\,
\nu_{\lambda,\epsilon}\leq H_\lambda^{-2\epsilon} ,
\max_{t\leq\nu_{\lambda,\epsilon}} |\ux(t)|\leq 2\theta_\lambda \bigr]
\end{align*}
and 
\begin{multline*}
\hat{\mathbf{P}}^{\ur}_\lambda (\ux^\lambda(1) \in A , \tau>1)
=
(c_4+o(1))\, h_\lambda^p\, \int_A u_W(z) e^{-|z|^2/2} \dd z \\
\times
\hat{\mathbf{E}}^{\ur}_\lambda \bigl[ 
u_W(\ux^\lambda(\nu_{\lambda,\epsilon})) \,;\,
\nu_{\lambda,\epsilon}\leq H_\lambda^{-2\epsilon} ,
\max_{t\leq\nu_{\lambda,\epsilon}} |\ux(t)|\leq 2\theta_\lambda \bigr] .
\end{multline*}
Thus, the proof of the first statement is completed.

To prove the functional convergence, it suffices to repeat the proof 
of~\cite[Theorem~1]{DW15b} using~\eqref{inv.3} and~\eqref{inv.4} instead 
of the corresponding estimates therein.
\end{proof}

\begin{corollary}
\label{inv.corr.1}
Let $W$ be the chamber of type $C$. If $\ur=\ur_\lambda \to \ur^*\in\partial W$,
then the sequence
$\hat{\mathbf{P}}^{\ur}_\lambda (\ux^\lambda(1)\in A \,|\, \tau>1)$
converges weakly. The densities of limiting laws on $W$ has are uniformly 
bounded. Moreover, $\ux^\lambda$ converges weakly on $C[0,1]$ towards the 
Brownian meander in $W$ started at $\ur^*$.
\end{corollary} 
\begin{proof}
We just split the original set of random walks into a finite number of subsets 
in such a way that the differences of coordinates of the starting points in 
every block converge to zero and the differences of coordinates from different
blocks stay bounded away from zero. Then, the probability that different blocks
do not intersect is bounded away from zero and, consequently, the conditioning 
on $\{\tau>1\}$ is equivalent to conditioning every block on staying in the
corresponding chamber. (If $r^*_1>0$, then every block is a random block in a 
chamber of type $A$, while if $r^*_1=0$, then the lowest block is a random walk 
in a chamber of type $C$ and all other blocks are random walks in chambers of 
type $A$.)
\end{proof}

\medskip
\noindent
\begin{proof}[Proof of~\eqref{eq:Inp2}]
Assume that there exists a sequence $\ur(j)$ such that
\[
\hat{\mathbf{P}}^{\ur(j)}_{1,\lambda} \bigl(
\max_{t\leq 1} x^\lambda_n(t)\leq 2\eta \bigm| \tau>1 \bigr) \to 0 .
\]
Since we are looking at starting points $\ur$ with $r_n\leq \eta$, there exists 
a convergent subsequence $\ur(j_k)$. Let $\ur^*$ denote the limiting point. It
follows immediately from the usual functional CLT that the case $\ur^*\in W$ is 
impossible. But, if $\ur^*\in\partial W$, then we may use 
Corollary~\ref{inv.corr.1} to conclude that the Brownian meander in $W$ started 
at $\ur^*$ leaves the set $\{\ux\in W \,:\, x_n\leq 2\eta\}$ with probability
one. However, this would contradict~\cite[Theorem~3.2]{KS11}.
Thus,
\begin{equation}
\label{inv.7}
\inf_{\ur \,:\, r_n\leq\eta} \hat{\mathbf{P}}^{\ur}_{1,\lambda}
\bigl( \max_{t\leq 1} x^\lambda_n(t)\leq 2\eta \bigm| \tau>1 \bigr) > 0 ,
\end{equation}
which implies~\eqref{eq:Inp2}.
\end{proof}

\begin{proof}[Proof of \eqref{eq:Inp3}]
Fix some $\epsilon>0$ and define
\[
W_{\leq\epsilon} =
\{\ux\in W \,:\, |x_{i+1}-x_i|\leq\epsilon \text{ for some } i\geq 0 \} .
\]
Assume that there exists a sequence $\ur(j)$ such that
\[
\hat{\mathbf{P}}^{\ur(j)}_{1,\lambda}
\bigl(\ux^\lambda(1)\in W_{\leq\epsilon} \,;\,
\max_{t\leq 1} x_n^\lambda(t)\leq 2\eta \bigm| \tau>1 \bigr)
\geq 
\epsilon^{1/2} .
\]
We may again assume that $\ur(j)$ converges to $\ur^*$ and this limiting point 
can not lie in $W$.
But, if $\ur^*$ is on the boundary of $W$, then the conditions of 
Corollary~\ref{inv.corr.1} are satisfied and the contradiction follows now 
from the boundedness of the density of the limiting law and the fact that
$vol(W_{\leq\epsilon} \cap \{\ux \,:\, x_n\leq 2\eta\})
\leq C_3 \eta^{n-1}\epsilon$.

As a consequence we have that, for all $\epsilon$ small enough, 
\begin{equation}
\label{inv.8}
\inf_{\ur \,:\, r_n\leq\eta}
\hat{\mathbf{P}}^{\ur}_{1,+,\lambda}
\bigl(\ux^\lambda(1)\in\Anlr , \max_{t\leq 1} x^\lambda_n(t)\leq 2\eta
\bigm| \tau>1 \bigr)
\geq 1- \epsilon^{1/2} .
\end{equation}
Combining~\eqref{inv.7} and~\eqref{inv.8}, we conclude that~\eqref{eq:Inp3}
holds for $t=1$ and all $\epsilon$ sufficiently small.
Using Brownian scaling, we conclude that~\eqref{eq:Inp3} is valid for all $t>0$.
\end{proof}

\bibliographystyle{plain}
\bibliography{OWalks}

\end{document}